\documentclass[11pt]{article}
\usepackage{amssymb}
\usepackage[usenames,dvipsnames]{xcolor}
\usepackage{amsthm}
\usepackage{amsmath}
\usepackage{amsfonts}

\usepackage{graphicx}

 \newtheorem{thm}{Theorem}[section]

 \theoremstyle{definition}
 
 \newtheorem{rem}[thm]{Remark}
 \numberwithin{equation}{section}

\newtheorem{theorem}{Theorem}[section]
\newtheorem{lemma}[theorem]{Lemma}

\theoremstyle{definition}

\theoremstyle{remark}

\begin{document}

\title{Uniform Bound of the Highest Energy for the 3D
Incompressible Elastodynamics}
\author{ Zhen Lei\footnote{School of Mathematical Sciences; LMNS and
Shanghai Key Laboratory for Contemporary Applied Mathematics,
Fudan University, Shanghai 200433, P. R. China. {\it Email:
leizhn@gmail.com, zlei@fudan.edu.cn.}} \ \ and Fan Wang\footnote{School of
Mathematical Sciences, Fudan University, Shanghai 200433, P. R. China.
{\it Email: wangf767@gmail.com.}}}
\date{}
\maketitle

\begin{abstract}
This article concerns the time growth of Sobolev norms of classical
solutions to the 3D incompressible isotropic elastodynamics with small
initial displacements. Given initial data in $H^k_\Lambda$ for a fixed
big integer $k$, the global well-posedness of this Cauchy problem
has been established by Sideris and Thomases in \cite{Sideris_2005}
and \cite{SiderisB_2005, SiderisB_2007}, where the highest-order
generalized energy $E_k(t)$ may have a certain growth in time.
Alinhac \cite{S. Alinhac_2010} conjectured that such a growth in
time may be a true phenomenon, where he proved that
$E_k(t)$ is still uniformly bounded in time only for 3D scalar quasilinear wave
equation under null condition. In this paper, we show that the highest-order
generalized energy $E_k(t)$ is still uniformly bounded for
the 3D incompressible isotropic elastodynamics. The equations
of incompressible elastodynamics can be viewed as
a nonlocal systems of wave type and is inherently linearly
degenerate in the isotropic case. There are three
ingredients in our proof: The first one is that
we still have a decay rate of $t^{- \frac{3}{2}}$ when
we do the highest energy estimate away from the light cone
even though in this case the Lorentz invariance is not available.
The second one is that the $L^\infty$ norm of the good unknowns,
in particular, $\nabla(v + G\omega)$, is shown to have a decay
rate of $t^{- \frac{3}{2}}$ near the light cone. The third one is that
the pressure is estimated in a novel way as a nonlocal nonlinear term with
null structure, as has been recently observed in \cite{LZT_2012}.
The proof employs the generalized energy method of Klainerman,
enhanced by weighted $L^2$ estimates and the ghost weight introduced by Alinhac.

\end{abstract}


\section{Introduction}%----------------------------introduction
This article considers the time growth of Sobolev norms of classical
solutions to the Cauchy problem of the 3D incompressible isotropic elastodynamics.
The equations of incompressible elastodynamics display a linear degeneracy in the isotropic
case; i.e., the equation inherently satisfies a null condition. By virtue of this nature, in a series of seminal
works in \cite{Sideris_2005} and \cite{SiderisB_2005, SiderisB_2007}, Sideris and Thomases
proved the global well-posedness of classical
solutions to the 3D incompressible isotropic elastodynamics with small initial displacements
for initial data in $H^k_\Lambda$ for some fixed big integer $k$ (notations will be introduced at
the beginning of Section 2). In those papers the highest-order generalized
energy $E_k(t)$ are shown to have an upper bound depending on time.

This kind of time growth of the highest Sobolev norms of the generalized energy also appears
in the work of Sideris \cite{Sideris_1996, Sideris_2000} and Agemi \cite{Agemi00} for the 3D compressible nonlinear elastic waves,
of Sideris and Tu \cite{Sideris_2002} for the 3D nonlinear wave systems
 and of Alinhac  \cite{S. Alinhac_2001} for the 2D scalar nonlinear wave equation (see also
 \cite{K. Hidano_2004, K. Hidano_20041, A. Hoshiga_2000, S. Katayama_1993, LeiLZ2011} for further results).
In \cite{S. Alinhac_2010},
Alinhac has proved that for 3D scalar quasilinear wave equation with small initial data under null condition, the energy is uniformly
bounded, and he also conjectured that by analogy with similar problems where such a growth has been
proved, that this time growth of the highest energy is a true phenomenon.

In this paper, we show that the highest-order generalized energy $E_k(t)$ is still uniformly bounded for
the Cauchy problem of the 3D incompressible isotropic elastodynamics with small initial data.
As a byproduct, our method presented here also provides a new and simpler proof for global well-posedness of
this problem. Our main theorem is as follows:

\begin{thm}\label{Main}
Let $(v_0, G_0)\in H^k_\Lambda$, with $k\geq 9$. Suppose that $(v_0, F_0) = (v_0, I+G_0)$
satisfy the constraints
\eqref{ConstraintG1} \eqref{ConstraintG2}, and $\|(v_0, G_0)\|_{H^k_\Lambda}<\epsilon$. Then
there exist two positive constants $M$ and $\epsilon_0$ which depend only
on $k$ such that, if $\epsilon\leq\epsilon_0$, then the system of incompressible
Hookean elastodynamics
\eqref{EquationM} with initial data $(v_0, F_0) = (v_0, I+G_0)$ has a unique global solution $(v, F) = (v, I+G)$
which satisfies $(v, G)\in H_\Gamma^k$ and $E_k^{1/2}(t)\leq M\epsilon$ for all $t\in [0, +\infty)$.
\end{thm}

The theorem is also true for general incompressible isotropic elastodynamics, which can be regarded as
a higher order nonlinear correction to the Hookean case. See Section 10 for more details.

We recall that the global existence of nonlinear wave and elastic systems in dimension two or three,
if it is true, hinges on two basic assumptions: the smallness of the initial data and the null
condition of nonlinearities. The omission of either of these assumptions can lead to the break down of
solutions in finite time \cite{F. John_19841, TahZ_1998}. For 3D quasilinear wave equation
whose quadratic nonlinearities satisfy the null condition, the global well-posedness
theory for small initial data was shown independently
by Christodoulou \cite{Christodoulou_1986} and Klainerman \cite{Klainerman_1986}.
The result of Christodoulou \cite{Christodoulou_1986} relied on the
conformal method, the null condition implies then that the nonlinear terms of
the equations transform into smooth terms, and the problem is reduced to a
local problem with small data. While the proof of Klainerman uses a special energy
inequality for the wave equation, which is obtained by multiplying by an appropriate
vector field with quadratic coefficients, which we called the generalized energy method,
see also \cite{Hormander_1997} for an account of both aspects.
The generalized energy method of Klainerman can be refined
to prove the global existence for compressible elastic waves with small initial data under the null condition \cite{Sideris_1996, Sideris_2000},
3D quasilinear wave equations with small initial data under the null condition \cite{Sideris_2002}, and incompressible elastic fluids with small initial data
\cite{Sideris_2005, SiderisB_2005, SiderisB_2007, Kessenich_2009}. Alinhac \cite{S. Alinhac_2001} can even enhance Klainerman's
generalized energy method to prove the global well-posedness of 2D quasilinear wave equation with small initial data under the null condition via a ghost weighted method.
However, in all the aforementioned results, the upper bound of the highest generalized energy depends on time.

The incompressible elasticity equation can be viewed as
a nonlocal system of wave type and is inherently linearly degenerate in the isotropic case.
Recently, a significant progress has been made in
\cite{LZT_2012} where the authors proved  the almost global existence of classical solutions for the 2D
incompressible isotropic elastodynamics by combining the generalized energy method
of Klainerman,  the weighted $L^2$ estimates of Sideris and the ghost weight method of Alinhac.
A key point in \cite{LZT_2012}, among other things, is that one can still use the null condition in the highest
order energy estimate. Motivated by \cite{LZT_2012}, in this paper, we show a uniform bound for the highest
generalized energy of solutions to the Cauchy problem of 3D incompressible isotropic elastodynamics
with small initial displacements. Let us mention that the final landmark work is the global wellposedness of the 2d incompressible elastodynamics, which is highly nontrivial and was recently solved by Lei in \cite{LeiZ13}.

There are three
ingredients in our proof: The first one is that
we still have a decay rate of $t^{- \frac{3}{2}}$ when we do the highest energy estimate away from the light cone
even though in this case the Lorentz invariance is not available. This is achieved by a refined Sobolev
embedding theorem which enables us to gain one spatial derivative. This one extra derivative allows us to obtain the 
$L^2$norm of $\langle t\rangle^{\frac{3}{2}}\nabla^2 U$ away from the light cone. See Lemma \ref{Genergy} for details.
And we also need to use the weighted
generalized energy defined by Sideris in \cite{Sideris_2005}. The second one is that
the $L^\infty$ norm of the good unknowns, in particular, $\nabla(v + G\omega)$, is shown to have a decay
rate of $t^{- \frac{3}{2}}$ near the light cone.
Let us mention that
%We also notice that by Lemma \ref{Decayr3/2}, which was proved by
%P. Kessenich in \cite{Kessenich_2009},
the special quantities
$\Gamma^\alpha v w$, $w\Gamma^\alpha G$ are also shown to have $t^{-3/2}$ time decay
in the region $r\geq \frac{\langle t\rangle}{16}$.
The third one is that
the pressure is estimated in a novel way as a nonlocal nonlinear term with
null structure, as has been recently observed in \cite{LZT_2012} in the two-dimensional case.
The proof also employs the generalized energy method
of Klainerman, enhanced by weighted $L^2$ estimates and the ghost weight introduced by Alinhac.

This paper is organized as follows: In Section 2 we  introduce some notations
and the system of isotropic elastodynamics. In Section 3 we discuss the commutation properties of
modified Klainerman's vector fields with equations and constraints. Then some calculus inequalities are
presented in Section 4. In Section 5, we will treat the pressure term. And in Section 6,
we show that the weighted generalized energy is controlled by the generalized energy for small solutions. In Section
7, we prove that the good unknown
$\partial_r(\Gamma^\alpha v+\Gamma^\alpha Gw)$ has a better decay
estimate \eqref{DecayOS2},
which gives $t^{-3/2}$ time decay in $L^\infty(r\geq \frac{\langle t \rangle}{16})$.
And in Section 8, we obtain $t^{-3/2}$ time decay in $L^\infty(r\leq \frac{\langle t \rangle}{16})$
for the derivative of the solution.
In Section 9 we complete the proof of the main result via the ghost weight
method introduced first by Alinhac in \cite{S. Alinhac_2001}.
Finally, we treat the general isotropic case by
regarding it as a higher order nonlinear correction to the Hookean case (in Section 10).

\section{Preliminaries}
Classically the motion of an elastic body is described as a
second-order evolution equation in Lagrangian coordinates. In the
incompressible case, the equations are more conveniently written as a first-order system
with constraints in Eulerian coordinates. We start with a
time-dependent family of orientation-preserving diffeomorphisms
$x(t, \cdot)$, $0 \leq t < T$. Material points $X$ in the
reference configuration are deformed to the spatial position $x(t,
X)$ at time $t$.
Let $X(t, x)$ be the corresponding reference map: $X(t, x)$ is the inverse of $x(t, \cdot).$
\begin{lemma}
Given a family of deformations $x(t, X)$, define the velocity and
deformation gradient as follows:
\begin{equation}
v(t, x) = \frac{dx(t, X)}{dt}\Big|_{X = X(t, x)},\quad F(t, x) =
\frac{\partial x(t, X)}{\partial X}\Big|_{X = X(t, x)}.
\end{equation}
Then for $0 \leq t < T$, $i, j, k \in \{1, 2, 3\}$, we
have
\begin{equation}
\partial_t F+v\cdot \nabla F=\nabla vF,
\end{equation}
\begin{equation}\label{ConstraintG1}
 \partial_jG_{ik} - \partial_kG_{ij}
= G_{mk}\partial_mG_{ij} - G_{mj}\partial_mG_{ik},
\end{equation}
where $G(t, x) = F(t, x) - I$. If in addition $x(t, X)$ is incompressible, that is $\det F(t, x) \equiv 1$, then
\begin{equation}\label{ConstraintG2}
\nabla\cdot F^T=\partial_mF_{mi} = 0 \quad \text{and}\quad\nabla\cdot v=0.
\end{equation}
\end{lemma}
\begin{proof}
\eqref{ConstraintG1} is proved in \cite{LeiLZ08}, for other identity, see for instance \cite{Sideris_2005, SiderisB_2007}.
\end{proof}
Here and in what follows, we use the summation convention over repeated indices.

In this paper, to best illustrate our methods and ideas, let us first consider
the equations of motion for incompressible Hookean elasticity,
which corresponds to the Hookean strain energy function $W(F) =
\frac{1}{2}|F|^2$ and reads
\begin{equation}\label{EquationM}
\begin{cases}
\partial_tv + v\cdot\nabla v + \nabla p = \nabla\cdot(FF^T),\\
\partial_tF + v\cdot\nabla F = \nabla vF,\\
\nabla\cdot v = 0,\ \ \ \ \nabla\cdot F^T=0.
\end{cases}
\end{equation}
As will be seen in Section 10 where the case of general energy function
is discussed, there is no essential loss of generality in considering this
simplest case. Throughout this paper we will adopt the notations of
$$(\nabla \cdot F)_i=\partial_jF_{ij},\quad (\nabla v)_{ij}=\partial_j v_i.$$

Most of our norms will
be in $L^2$ and most integrals will be taken over $\mathbb{R}^3$, so we write
\begin{equation}
\begin{split}
\|\cdot\|=\|\cdot\|_{L^2(\mathbb{R}^3)},
\end{split}
\end{equation}
and
\begin{equation}
\int=\int_{\mathbb{R}^3}.
\end{equation}
We use the usual derivative vector fields
\begin{equation}
\partial=(\partial_t, \partial_1, \partial_2, \partial_3)\ \ \ \text{and} \ \ \ \nabla= ( \partial_1, \partial_2, \partial_3).
\end{equation}
The scaling operator is denoted by:
\begin{equation}
S=t\partial_t+\sum_{j=1}^3 x_j\partial_j =t\partial_t + r\partial_r,
\end{equation}
and its time independent analogue is
\begin{equation}
S_0= r\partial_r,
\end{equation}
here, the radial derivative is defined by $\partial_r=\frac{x}{r}\cdot \nabla,\ \ r=|x|$.
The angular momentum operators are defined by
\begin{equation}
\Omega = x\wedge \nabla.
\end{equation}

In this paper, the notation $U= (v, G)$ is used frequently, where $v$ are three dimensional
real valued vector functions
and $G$ are three by three real valued matrix functions. The arguments of these functions are
nearly always suppressed.
As in \cite{Sideris_2005, SiderisB_2007} and \cite{Kessenich_2009}, we define
\begin{equation*}
\widetilde{\Omega}_i(U)= (\Omega_i G+ [V^{i}, G], \Omega_i v+V^{i}v),
\end{equation*}
where
\begin{equation}
\begin{split}
&V^{1}=e_2\otimes e_3- e_3\otimes e_2,\\
&V^{2}=e_3\otimes e_1- e_1\otimes e_3,\\
&V^{3}=e_1\otimes e_2- e_2\otimes e_1,\\
\end{split}
\end{equation}
and $[A,B] = AB - BA$ denotes the standard Lie bracket product.
Occasionally we will write
\begin{equation}
\widetilde{\Omega} G= \Omega G+ [V, G]\ \ \text{for}\ \ G\in \mathbb{R}^3\otimes \mathbb{R}^3,
\end{equation}
and
\begin{equation}
\widetilde{\Omega} v= \Omega v+ V v\ \ \text{for}\ \ v\in \mathbb{R}^3.
\end{equation}
For scalar functions $f$ we define
\begin{equation}
\widetilde{\Omega} f= \Omega f.
\end{equation}
We shall frequently use the decomposition
\begin{equation}\label{Decompose}
\nabla=\frac{x}{r}\partial_r-\frac{x}{r^2}\wedge \Omega.
\end{equation}
Our vector fields will be written succinctly as $\Gamma$. We let
\begin{equation}
\Gamma = (\Gamma_1, \cdots, \Gamma_8) = (\partial, \widetilde{\Omega}, S).
\end{equation}
Hence by $\Gamma U$ we mean any one of $\Gamma_i U$. By $\Gamma^a$, $a= (a_1, \cdots, a_\kappa)$,
we denote an ordered product of $k=|a|$ vector fields $\Gamma_{a_1}\cdots \Gamma_{a_k}$, we note
that the commutator of any $\Gamma^{,}s$ is again a $\Gamma$.

Define the generalized energy by
\begin{equation}
E_k(t) = \sum_{|\alpha|\leq k}\big(\|\Gamma^\alpha v(t, \cdot)\|^2 + \|\Gamma^\alpha G(t, \cdot)\|^2 \big).
\end{equation}
We also define the weighted energy norm
\begin{equation}
X_k(t) = \sum_{|\alpha|\leq k-1}\big(\|\langle t-r \rangle\nabla\Gamma^\alpha v\|^2+
\|\langle t-r \rangle\nabla\Gamma^\alpha G\|^2 \big),
\end{equation}
in which we denote $\langle \sigma \rangle=\sqrt{1+ \sigma^2}$.

In order to characterize the initial data, we introduce the time independent analogue of $\Gamma$.
The only difference will be in the scaling operator. Set
\begin{equation}
\Lambda=(\Lambda_1, \cdots, \Lambda_7) = (\nabla, \widetilde{\Omega}, S_0).
\end{equation}
Then the commutator of any two $\Lambda^,s$ is again a $\Lambda$.
Define
\begin{equation}
H_\Lambda^k = \{U=(v, G): \mathbb{R}^3\rightarrow \mathbb{R}^3\times (\mathbb{R}^3\otimes \mathbb{R}^3):
\sum_{|\alpha|\leq k}\|\Lambda^\alpha U\|<\infty\}.
\end{equation}
Solutions will be constructed in the space
\begin{equation}
\begin{split}
H_\Gamma^k = \{&U=(v, G): [0, \infty)\times\mathbb{R}^3\rightarrow
\mathbb{R}^3\times (\mathbb{R}^3\otimes \mathbb{R}^3):\\
&(v, G)\in \cap_{j=0}^kC^j([0, \infty); H_\Lambda^{k-j})\}.
\end{split}
\end{equation}

\section{Commutation}
Write $F=I+G$, we obtain the following systems:
\begin{equation}\label{EquationG}
\begin{cases}
\partial_tv -\nabla\cdot G   =-\nabla p-v\cdot\nabla v+ \nabla\cdot(GG^T),\\
\partial_tG -\nabla v=- v\cdot\nabla G +\nabla vG,\\
\nabla\cdot v = 0,\ \ \ \nabla\cdot G^T = 0,
\end{cases}
\end{equation}
where $G$ satisfies the following equation
\begin{equation}\label{ConstraintG0}
\partial_k G_{ij}-\partial_j G_{ik}=G_{lj}\partial_l G_{ik}-G_{lk}\partial_l G_{ij}.
\end{equation}
For any multi-index $\alpha$, apply
$\Gamma^\alpha$ to the equation, we have the following  commutation properties, see, for instance \cite{LZT_2012}.
\begin{equation}\label{Equation}
\begin{cases}
\partial_t\Gamma^\alpha v - \nabla\cdot\Gamma^\alpha G =  - \nabla\Gamma^\alpha p + f_\alpha, \\
\partial_t\Gamma^\alpha G  - \nabla\Gamma^\alpha v = g_\alpha,\\
\nabla\cdot\Gamma^\alpha v = 0,\quad \nabla\cdot\Gamma^\alpha G^T
= 0,
\end{cases}
\end{equation}
where
\begin{equation}\label{Equationfg}
\begin{cases}
f_\alpha =-\sum_{\beta + \gamma = \alpha}\Gamma^\beta v\cdot\nabla \Gamma^\gamma v
+\sum_{\beta + \gamma = \alpha}\nabla\cdot(\Gamma^\beta G \Gamma^\gamma G^T),\\
g_\alpha = - \sum_{\beta + \gamma = \alpha}
\Gamma^\beta v\cdot\nabla \Gamma^\gamma
G + \sum_{\beta + \gamma = \alpha}\nabla\Gamma^\beta v \Gamma^\gamma G,
\end{cases}
\end{equation}
we also have
\begin{equation}\label{ConstraintG}
\partial_k\Gamma^\alpha G_{ij} - \partial_j\Gamma^\alpha G_{ik} = h_\alpha,
\end{equation}
where
\begin{equation}\nonumber
h_\alpha = \sum_{|\beta| + |\gamma| = |\alpha|}
\Gamma^\beta G_{lj}\partial_l\Gamma^\gamma
G_{ik} - \Gamma^\beta
G_{lk}\partial_l\Gamma^\gamma G_{ij}.
\end{equation}

\section{Calculus Inequalities}
In this section, a few elementary but useful inequalities will be prepared. Such inequalities capture
decay at spatial infinity through the use of the vector field $\Gamma$. The first result appeared in
\cite{Sideris_2000}, and the second appeared in \cite{Sideris_2005, SiderisB_2007}. We will omit the proofs.

\begin{lemma}\label{Ginequality}
For $u\in C_0^\infty(\mathbb{R}^3)^3$, $r=|x|$, and $\rho= |y|$,
\begin{equation}
\begin{split}
r^{1/2}|u(x)| & \lesssim\sum_{|\alpha|\leq 1}\|\nabla \widetilde{\Omega}^\alpha u\|,\\
r|u(x)| & \lesssim\sum_{|\alpha|\leq 1} \|\partial_r \widetilde{\Omega}^\alpha u\|_{L^2(|y|\geq r)}^{1/2}\cdot
\sum_{|\alpha|\leq 2} \|\widetilde{\Omega}^\alpha u\|_{L^2(|y|\geq r)}^{1/2}.
\end{split}
\end{equation}
\end{lemma}
In this paper, we write $X\lesssim Y$ to indicate $X\leq CY$ for some constant $C> 0$,
and $X \sim Y$ whenever $X\lesssim Y\lesssim X$.
\begin{lemma}
Let $U\in H_\Gamma^k$, with $X_k[U(t)]<\infty$ and $|U|<\delta$ small. Then
we have
\begin{equation}\label{Ginequality2}
\begin{split}
&\langle r\rangle|\Gamma^\alpha U(t, x)|\lesssim E_k^{1/2}[U(t)],\quad |\alpha|+2\leq k,\\
&\langle r\rangle\langle t-r\rangle |\nabla\Gamma^\alpha U(t, x)|\lesssim X_k^{\frac{1}{2}}[U(t)],\quad |\alpha|+3\leq k.\\
\end{split}
\end{equation}
\end{lemma}
The following Sobolev inequality gives $(1+t)^{-1}$ time decay
in $L^\infty(r\leq \frac{\langle t \rangle}{16})$ by means of
the weighted Sobolev norms. Instead of using $|f|_{L^\infty}\lesssim\sum_{|\alpha|\leq 2}\|\nabla^\alpha f\|$
as in \cite{Sideris_2005, SiderisB_2007}, we use the inequality
$|f|_{L^\infty(\mathbb{R}^3)}\lesssim \|\nabla f\|^{1/2}\|\nabla^2 f\|^{1/2}$
for any $f\in H^2(\mathbb{R}^3)$, this enables us to gain one spatial derivative.
Therefore the weight appeared in the terms which have at least one spatial derivative.
In Section 8, we will prove that for $U=(v,G)$, the solution of \eqref{EquationG},
$\nabla U$ has $(1+t)^{-3/2}$ time decay in $L^\infty(r\leq \frac{\langle t \rangle}{16})$,
which is important for us to prove our main theorem.
\begin{lemma}\label{DecayT_0}
For all $f\in H^2(\mathbb{R}^3)$, there holds
\begin{equation}
\langle t\rangle |f|_{L^\infty(r\leq \frac{\langle t \rangle}{16})}\lesssim
\|f\|_{L^2(r\leq \frac{\langle t \rangle}{8})}+\|\langle t-r\rangle\nabla f\|_{L^2(r\leq \frac{\langle t \rangle}{8})}
+\|\langle t-r\rangle\nabla^2 f\|_{L^2(r\leq \frac{\langle t \rangle}{8})},
\end{equation}
provide the right hand side is finite.
\end{lemma}
\begin{proof}
Let $\varphi\in C_0^\infty$, satisfy $\varphi(s)=1$ for $s\leq 1$, $\varphi(s)=0$ for $s\geq 2$.
we have by Sobolev embedding $H^2(\mathbb{R}^3)\subset L^\infty(\mathbb{R}^3)$,
$|f|_{L^\infty(\mathbb{R}^3)}\lesssim \|\nabla f\|^{1/2}\|\nabla^2 f\|^{1/2}$,
\begin{equation}
\begin{split}
\langle t\rangle|f(x)|&= \langle t\rangle \varphi(r/\frac{\langle t \rangle}{16})|f(x)|\\
&\lesssim \langle t\rangle \|\nabla(\varphi(r/\frac{\langle t \rangle}{16})f)\|_{L^2(r\leq \frac{\langle t \rangle}{8})}+\langle t\rangle \|\nabla^2(\varphi(r/\frac{\langle t \rangle}{16})f)\|_{L^2(r\leq \frac{\langle t \rangle}{8})}\\
&\lesssim \|f\|_{L^2(r\leq \frac{\langle t \rangle}{8})}+\langle t\rangle\|\nabla f\|_{L^2(r\leq \frac{\langle t \rangle}{8})}+\langle t\rangle\|\nabla^2 f\|_{L^2(r\leq \frac{\langle t \rangle}{8})}\\
&\lesssim \|f\|_{L^2(r\leq \frac{\langle t \rangle}{8})}+\|\langle t-r\rangle\nabla f\|_{L^2(r\leq \frac{\langle t \rangle}{8})}+\|\langle t-r\rangle\nabla^2 f\|_{L^2(r\leq \frac{\langle t \rangle}{8})}.
\end{split}
\end{equation}
\end{proof}
\begin{rem}
For any $\alpha$, we have
\begin{equation}\label{DecayT}
\langle t\rangle |\Gamma^\alpha f|_{L^\infty(r\leq \frac{\langle t \rangle}{16})}\lesssim E_{|\alpha|}^{1/2}+X_{|\alpha|+2}^{1/2}.
\end{equation}
\end{rem}

The following result is a consequence of Corollary 8.1 in \cite{Kessenich_2009},
which states that the special quantities $w\Gamma^\alpha G$, $w\Gamma^\alpha v$ have $(1+t)^{-3/2}$
time decay outside the light cone.
\begin{lemma}\label{Decayr3/2}
Suppose $\nabla\cdot v=0$, $\nabla\cdot G^T=0$, then we have
\begin{equation}
\begin{split}
&r^{3/2}|w\cdot v|\lesssim\sum_{|\alpha|\leq 2}\|\Gamma^{\alpha}v\|,\\
&r^{3/2}|w_i\cdot G_{ij}|\lesssim\sum_{|\alpha|\leq 2}\|\Gamma^{\alpha}G\|, \forall j=1, 2, \cdots, n.
\end{split}
\end{equation}
\end{lemma}
\begin{proof}
For completeness, let us give a simpler proof. It suffices to prove the first one, since the second one can be carried out
in exactly the same way.

For $r\leq 1$, it is an immediate consequence of the Sobolev embedding
\begin{equation*}
H^2(\mathbb{R}^3)\subset L^\infty(\mathbb{R}^3).
\end{equation*}

While for $r\geq 1$, using Lemma \ref{Ginequality}, and the decomposition of $\nabla$,
we have
\begin{equation}
\begin{split}
r^{\frac{1}{2}}|rv\cdot w|&\lesssim\sum_{|\alpha|\leq 1}\|\nabla \widetilde{\Omega}^\alpha(rv\cdot w)\|=
\sum_{|\alpha|\leq 1}\|\nabla (r\widetilde{\Omega}^\alpha v\cdot w)\|\\
&\lesssim\sum_{|\alpha|\leq 1}\|\partial_r (r\widetilde{\Omega}^\alpha v\cdot w)\|+
\sum_{|\alpha|\leq 1}\|\frac{w\wedge \Omega}{r} (r\widetilde{\Omega}^\alpha v\cdot w)\|\\
&\lesssim\sum_{|\alpha|\leq 2}\|\widetilde{\Omega}^\alpha v\|+
\sum_{|\alpha|\leq 1}\|r\partial_r\widetilde{\Omega}^\alpha v\cdot w\|\\
&\lesssim\sum_{|\alpha|\leq 2}\|\widetilde{\Omega}^\alpha v\|+
\sum_{|\alpha|\leq 1}(\|r\partial_i\widetilde{\Omega}^\alpha v_i\|+
\|r\frac{(w\wedge \Omega)_i}{r}\widetilde{\Omega}^\alpha v_i\|)\\
&\lesssim\sum_{|\alpha|\leq 2}\|\Gamma^\alpha v\|.
\end{split}
\end{equation}
\end{proof}

\section{Bound for the Pressure Gradient}
The following Lemma shows that the pressure gradient may be
treated as a nonlinear term. The first estimate appeared in \cite{SiderisB_2007},
and the second estimate appeared in \cite{LZT_2012}, it is a novel refinement which saves one derivative over
the first bound and which allows us to exploit the null structure.
This is essential in Section 9 when we estimate the ghost weighted energy.
\begin{lemma}\label{Pressure}
Let $(v, F)=(v, I+G)$, $(v, G) \in H^k_\Gamma$, solve the equation \eqref{EquationM},
%and satisfy the constraint \eqref{Consrtaint}.
then we have
\begin{equation}
\begin{split}
&\|\nabla\Gamma^\alpha p\| \lesssim \|f_\alpha\|,\\
\end{split}
\end{equation}
\begin{equation}
\|\nabla\Gamma^\alpha p\| \lesssim \sum_{\beta+\gamma=\alpha, |\beta|\leq|\gamma|}
\|\partial_j\Gamma^\beta v_i\Gamma^\gamma v_j-\partial_j\Gamma^\beta G_{ik}\Gamma^\gamma G_{jk}\|,
\end{equation}
for all $|\alpha|\leq k-1$.
\end{lemma}
\begin{proof}
Applying the divergence operator to the first equation of
\eqref{Equation}, we have
\begin{equation}
\Delta \Gamma^\alpha p = \nabla\cdot f_\alpha + \nabla\cdot(\nabla\cdot
\Gamma^\alpha G) - \partial_t\nabla\cdot \Gamma^\alpha v.
\end{equation}
Using the last equation in \eqref{Equation}, one has
\begin{equation}
\Delta \Gamma^\alpha p = \nabla\cdot f_\alpha.
\end{equation}
By the definition of $f_\alpha$ \eqref{Equationfg}, and the last equation in \eqref{Equation}, we have
\begin{equation}
\begin{split}
\nabla\cdot f_\alpha &= -\sum_{\beta+\gamma=\alpha}\partial_i\partial_j(\Gamma^\beta v_i \Gamma^\gamma v_j -\Gamma^\beta G_{ik} \Gamma^\gamma G_{jk})\\
&= -\sum_{\beta+\gamma=\alpha, |\beta|\leq|\gamma|}\partial_i\partial_j(\Gamma^\beta v_i \Gamma^\gamma v_j -\Gamma^\beta G_{ik} \Gamma^\gamma G_{jk})\\
&\ \ \ \ -\sum_{\beta+\gamma=\alpha, |\beta|>|\gamma|}\partial_i\partial_j(\Gamma^\beta v_i \Gamma^\gamma v_j -\Gamma^\beta G_{ik} \Gamma^\gamma G_{jk})\\
&= -\sum_{\beta+\gamma=\alpha, |\beta|\leq|\gamma|}\partial_i(\partial_j\Gamma^\beta v_i \Gamma^\gamma v_j -\partial_j\Gamma^\beta G_{ik} \Gamma^\gamma G_{jk})\\
&\ \ \ \ -\sum_{\beta+\gamma=\alpha, |\beta|>|\gamma|}\partial_j(\Gamma^\beta v_i \partial_i\Gamma^\gamma v_j -\Gamma^\beta G_{ik} \partial_i\Gamma^\gamma G_{jk}).
\end{split}
\end{equation}
The result now follows since
\begin{equation}
\nabla \Gamma^\alpha p=\Delta^{-1}\nabla(\nabla\cdot f_\alpha),
\end{equation}
and since $\Delta^{-1}\nabla\otimes \nabla$ is bounded in $L^2$.
\end{proof}

\section{Weighted $L^2$ Estimate}
In this section, we show that the weighted norm is controlled by the energy, for small solutions.
\begin{lemma}\label{Lemma6.1}
Suppose that $(v, F)=(v, I+G)$, $(v, G)\in H_\Gamma^k, k\geq 4$, solves \eqref{EquationM}.
Then
\begin{equation}
\|N_k(t)\|\lesssim E_k(t) + E_k(t)^{1/2} X_k(t)^{1/2},
\end{equation}
there
\begin{equation}
N_k=\sum_{|\alpha|\leq k-1}\big[ t|f_\alpha|+ t|g_\alpha|+ (t+r)|h_\alpha|+t|\nabla \Gamma^\alpha p|\big].
\end{equation}
\end{lemma}
\begin{proof}
By Lemma \ref{Pressure}, we have
\begin{equation}
\begin{split}
\|N_k(t)\|&\leq \sum_{|\alpha|\leq k-1}\big[ t\|\nabla \Gamma^\alpha p\|+t\|f_\alpha\|+ t\|g_\alpha\|+
\|(t+r)h_\alpha\|\big]\\
&\lesssim \sum_{|\alpha|\leq k-1}\big[t\|f_\alpha\|+ t\|g_\alpha\|+
\|(t+r)h_\alpha\|\big].
\end{split}
\end{equation}

To estimate these terms, we shall consider the following two cases: $r\leq \frac{\langle t \rangle}{16}$ and $r\geq \frac{\langle t \rangle}{16}$.

{\bf Estimates of nonlinearities for $r\leq \frac{\langle t \rangle}{16}$}. Examining the definitions, we find that
\begin{equation}
\begin{split}
&\ \ \ \ \sum_{|\alpha|\leq k-1}\big[t\|f_\alpha\|_{L^2(r\leq \langle t \rangle/16)}+ t\|g_\alpha\|_{L^2(r\leq \langle t \rangle/16)}+
\|(t+r)h_\alpha\|_{L^2(r\leq \langle t \rangle/16)}\big]\\
&\lesssim  \sum_{\beta+\gamma=\alpha, |\alpha|\leq k-1}\langle t\rangle\|(|\Gamma^\beta v|+ |\Gamma^\beta G|)
(|\nabla \Gamma^\gamma v|+|\nabla \Gamma^\gamma G|)\|_{L^2(r\leq \langle t \rangle/16)}.
\end{split}
\end{equation}
To simplify the notation a bit, we shall write
\begin{equation}
|(\Gamma^k v, \Gamma^k G)|=\sum_{|\alpha|\leq k}[|\Gamma^\alpha v|+|\Gamma^\alpha G|].
\end{equation}
We make use of the fact that since $\beta+\gamma=\alpha$, $|\alpha|\leq k-1$ and $k\geq 4$,
either $|\beta|\leq k^{\prime}$ or $|\gamma| \leq k^{\prime}$, with $k^{\prime}=[\frac{k-1}{2}]$,
we have $k^\prime+3\leq k$. Here, $[s]$ stands for the largest integer not exceeding $s$.
Note that $\langle t\rangle\sim \langle t-r\rangle$ in the region $r\leq \frac{\langle t \rangle}{16}$, thus, using \eqref{DecayT}, we have
\begin{equation}
\begin{split}
&\ \ \ \ \langle t\rangle\|(|\Gamma^\beta v|+ |\Gamma^\beta G|)
(|\nabla \Gamma^\gamma v|+|\nabla \Gamma^\gamma G|)\|_{L^2(r\leq \langle t \rangle/16)}\\
&\lesssim |(|\Gamma^{k^{\prime}} v, \Gamma^{k^{\prime}} G)|_{L^\infty}\|
\langle t-r\rangle(\nabla \Gamma^{k-1} v, \nabla \Gamma^{k-1} G|)\|_{L^2(r\leq \langle t \rangle/16)}\\
&\ \ \ \ +\|(|\Gamma^{k-1} v, \Gamma^{k-1} G)\|_{L^2}|\langle t\rangle
(\nabla \Gamma^{k^{\prime}} v, \nabla \Gamma^{k^{\prime}} G|)|_{L^\infty(r\leq \langle t\rangle/16)}\\
&\lesssim E_{k^{\prime}+2}^{1/2}(t)X_k^{1/2}(t)+E_{k-1}^{1/2}(t)(E_{k^{\prime}+1}^{1/2}(t)+X_{k^{\prime}+3}^{1/2}(t))\\
&\lesssim E_k^{1/2}(t)X_k^{1/2}(t)+E_k(t).
\end{split}
\end{equation}

{\bf Estimates of nonlinearities for $r\geq \frac{\langle t \rangle}{16}$}. By the decomposition of
$\nabla$ \eqref{Decompose}, we have
\begin{equation}
\begin{split}
|f_\alpha|+|g_\alpha|+|h_\alpha| &\lesssim \sum_{\beta+\gamma=\alpha, |\alpha|\leq k-1}
(|\Gamma^\beta v|+|\Gamma^\beta G|))(|\partial_r\Gamma^\gamma v|+|\partial_r\Gamma^\gamma G|)\\
&+\sum_{\beta+\gamma=\alpha, |\alpha|\leq k-1}\frac{1}{r}(|\Gamma^\beta v|+|\Gamma^\beta G|)
(|\Omega\Gamma^\gamma v|+|\Omega\Gamma^\gamma G|).
\end{split}
\end{equation}
Thus, using the second inequality of \eqref{Ginequality}, we have
\begin{equation}
\begin{split}
&\ \ \ \ \sum_{|\alpha|\leq k-1}\big[t\|f_\alpha\|_{L^2(r\geq \langle t \rangle/16)}+ t\|g_\alpha\|_{L^2(r\geq \langle t \rangle/16)}+
\|(t+r)h_\alpha\|_{L^2(r\geq \langle t \rangle/16)}\big]\\
&\lesssim \sum_{\beta+\gamma=\alpha, |\alpha|\leq k-1}
\|r(|\Gamma^\beta v|+|\Gamma^\beta G|))(|\partial_r\Gamma^\gamma v|+|\partial_r\Gamma^\gamma G|)\|_{L^2(r\geq \langle t \rangle/16)}\\
&\ \ \ \ +\sum_{\beta+\gamma=\alpha, |\alpha|\leq k-1}\|(|\Gamma^\beta v|+|\Gamma^\beta G|)
(|\Omega\Gamma^\gamma v|+|\Omega\Gamma^\gamma G|)\|_{L^2(r\geq \langle t \rangle/16)},\\
&\lesssim E_k(t).
\end{split}
\end{equation}
Then Lemma \ref{Lemma6.1} follows by collecting the estimates above.
\end{proof}
\begin{lemma}\label{GequationG}
Suppose that $(v, F)=(v, I+G)$, $(v, G)\in H_\Gamma^k$, solves \eqref{EquationM}.
Then for all $|\alpha|\leq k-1$, we have
\begin{equation}
\begin{split}
\|(t-r)\nabla\Gamma^\alpha G\|^2\lesssim \|\Gamma^\alpha G\|^2+\|(t-r)\nabla\cdot \Gamma^\alpha G\|^2+Q_\alpha,
\end{split}
\end{equation}
here
\begin{equation}
\begin{split}
Q_\alpha=\sum_{\beta+\gamma=\alpha, |\alpha|\leq k-1}\int(t-r)^2\partial_j\Gamma^{\alpha}G_{ik}(\Gamma^{\beta}G_{lk}\partial_l \Gamma^{\gamma} G_{ij}-\Gamma^{\beta}G_{lj}\partial_l \Gamma^{\gamma}G_{ik}),
\end{split}
\end{equation}
if in addition $k\geq 4$, then we have
\begin{equation}
Q_\alpha\lesssim X_kE_k^{1/2}.
\end{equation}
\end{lemma}

\begin{proof}
Using integration by parts and \eqref{ConstraintG0}, combining with Young's inequality, we have
\begin{equation}
\begin{split}
\|(t-r)\nabla G\|^2&=\int(t-r)^2\partial_jG_{ik}\partial_jG_{ik}\\
&=\int(t-r)^2\partial_jG_{ik}\partial_kG_{ij}+Q\\
&=\int 2(t-r)w_jG_{ik}\partial_kG_{ij}-(t-r)^2G_{ik}\partial_j\partial_kG_{ij}+Q\\
&=\int 2(t-r)w_jG_{ik}\partial_kG_{ij}-2(t-r)w_kG_{ik}\partial_jG_{ij}\\
&\ \ \ \ +(t-r)^2\partial_kG_{ik}\partial_jG_{ij}+Q\\
&\leq 4\|G\|^2+\frac{3}{2}\|(t-r)\nabla\cdot G\|^2+\frac{1}{2}\|(t-r)\nabla G\|^2+Q,
\end{split}
\end{equation}
hence
\begin{equation}
\begin{split}
\|(t-r)\nabla G\|^2
\lesssim \|G\|^2+\|(t-r)\nabla\cdot G\|^2+Q,
\end{split}
\end{equation}
here
\begin{equation}
\begin{split}
Q=\int(t-r)^2\partial_jG_{ik}(G_{lk}\partial_l G_{ij}-G_{lj}\partial_l G_{ik}).
\end{split}
\end{equation}
Lemma \ref{GequationG} follows by applying this inequality to $\Gamma^\alpha G$. On the other hand,
since $\beta+\gamma=\alpha$, $|\alpha|\leq k-1$ , $\gamma\neq \alpha$ and $k\geq 4$,
either $|\beta|\leq k^{\prime}$ or $|\gamma| \leq k^{\prime}$, with $k^{\prime}=[\frac{k-1}{2}]$.
Hence, by Sobolev embedding $H^2(\mathbb{R}^3)\subset L^\infty(\mathbb{R}^3)$, Lemma \ref{Ginequality2}, we have

\begin{equation}
\begin{split}
&\ \ \ \ Q_\alpha=\sum_{\beta+\gamma=\alpha,|\alpha|\leq k-1}\int(t-r)^2\partial_j\Gamma^{\alpha}G_{ik}(\Gamma^{\beta}G_{lk}\partial_l \Gamma^{\gamma} G_{ij}-\Gamma^{\beta}G_{lj}\partial_l \Gamma^{\gamma}G_{ik})\\
&\lesssim \int\langle t-r\rangle^2|\nabla \Gamma^{k-1}G|(|\Gamma^{k^\prime}G\nabla \Gamma^{k-1}G|+|\Gamma^{k-1}G\nabla \Gamma^{k^\prime}G|)\\
&\lesssim \|\langle t-r\rangle\nabla \Gamma^{k-1}G\|^2|\Gamma^{k^\prime}G|_{L^\infty}
+\|\langle t-r\rangle\nabla \Gamma^{k-1}G\|\|\Gamma^{k-1}G\||\langle t-r\rangle\nabla \Gamma^{k^\prime}G|_{L^\infty}\\
&\lesssim X_kE_{k^\prime+2}^{1/2}+X_k^{1/2}E_{k-1}^{1/2}X_{k^\prime+3}^{1/2}\lesssim X_kE_k^{1/2}.
\end{split}
\end{equation}
\end{proof}
\begin{lemma}
Suppose that $(v, F)=(v, I+G)$, $(v, G)\in H_\Gamma^k$, $k\geq 1$, solves \eqref{EquationM}.
Define
\begin{equation}
L_k=\sum_{|\alpha|\leq k}\big[|\Gamma^\alpha v|+|\Gamma^\alpha G|\big],
\end{equation}
Then for all $|\alpha| \leq k-1$,
\begin{equation}\label{Equation1}
r|\nabla\cdot \Gamma^\alpha G\otimes w-\partial_r\Gamma^\alpha G|\lesssim L_k+N_k,
\end{equation}
\begin{equation}\label{Equation2}
r|\nabla\Gamma^\alpha v-\partial_r \Gamma^\alpha v\otimes w|\lesssim L_k,
\end{equation}
\begin{equation}\label{Equation3}
(t\pm r)|\nabla \Gamma^\alpha v\pm \nabla \cdot \Gamma^\alpha G\otimes w|\lesssim L_k+N_k.
\end{equation}
\end{lemma}
\begin{proof}
By the decomposition of $\nabla$, we have
\begin{equation}\label{H}
\begin{split}
&\ \ \ \ (\nabla\cdot \Gamma^\alpha G\otimes w)_{ij}=\partial_k\Gamma^\alpha G_{ik}w_j=w_k\partial_r\Gamma^\alpha G_{ik}w_j
-\frac{(w\wedge \Omega)_k}{r}\Gamma^\alpha G_{ik}w_j\\
&=\partial_j\Gamma^\alpha G_{ik}w_k+\frac{(w\wedge \Omega)_j}{r}\Gamma^\alpha G_{ik}w_k-\frac{(w\wedge \Omega)_k}{r}\Gamma^\alpha G_{ik}w_j.
\end{split}
\end{equation}
Notice that, by \eqref{ConstraintG}, we have
\begin{equation}
\begin{split}
\partial_j\Gamma^\alpha G_{ik}w_k=\partial_k\Gamma^\alpha G_{ij}w_k-h_\alpha w_k
=\partial_r\Gamma^\alpha G_{ij}-h_\alpha w_k,
\end{split}
\end{equation}
hence, we have
\begin{equation}
\begin{split}
r|\nabla\cdot \Gamma^\alpha G\otimes w-\partial_r\Gamma^\alpha G|\lesssim L_k+N_k,
\end{split}
\end{equation}
which is \eqref{Equation1}.

Similarly, by the decomposition of $\nabla$, we also have
\begin{equation}
\begin{split}
(\nabla \Gamma^\alpha v)_{ij}=\partial_j \Gamma^\alpha v_i=w_j\partial_r \Gamma^\alpha v_i-
\frac{(w\wedge \Omega)_j}{r}\Gamma^\alpha v_i,
\end{split}
\end{equation}
hence \eqref{Equation2} follows.

Using the equation \eqref{Equation} and the definition $S=t\partial_t + r\partial_r$, we can write
\begin{equation}\label{LastE}
\begin{split}
&t\nabla \cdot \Gamma^\alpha G+ r\partial_r \Gamma^\alpha v=S\Gamma^\alpha v-t f_\alpha +t\nabla \Gamma^\alpha p,\\
&t\nabla \Gamma^\alpha v+ r\partial_r \Gamma^\alpha G=S\Gamma^\alpha G-tg_\alpha.
\end{split}
\end{equation}
This is rearranged as follows:
\begin{equation}
\begin{split}
&t\nabla \cdot \Gamma^\alpha G\otimes w+r\nabla\Gamma^\alpha v=r\big[\nabla\Gamma^\alpha v-\partial_r \Gamma^\alpha v\otimes w\big]
+ \big[S\Gamma^\alpha v-t f_\alpha +t\nabla \Gamma^\alpha p\big]\otimes w,\\
&t\nabla \Gamma^\alpha v+ r\nabla\cdot\Gamma^\alpha G\otimes w=r\big[\nabla\cdot\Gamma^\alpha G\otimes w-\partial_r \Gamma^\alpha G\big]+S\Gamma^\alpha G-tg_\alpha.
\end{split}
\end{equation}
Then, combining with the inequality \eqref{Equation1} \eqref{Equation2}, we have
\begin{equation}
\begin{split}
(t\pm r)|\nabla \Gamma^\alpha v\pm \nabla \cdot \Gamma^\alpha G\otimes w|\lesssim L_k+N_k,
\end{split}
\end{equation}
which is inequality \eqref{Equation3}.
\end{proof}
\begin{lemma}\label{EX}
Suppose that $(v, F)=(v, I+G)$, $(v, G)\in H_\Gamma^k, k\geq 4$, solves \eqref{EquationM}.
If
$E_k(t)\ll 1$, then $X_k(t)\lesssim E_k(t)$.
\end{lemma}
\begin{proof}
 Starting with the definition of $X_k$, and using the fact that $\langle t-r\rangle\leq 1+|t-r|$,
 we obtain that
 \begin{equation}
 \begin{split}
X_k(t) &= \sum_{|\alpha|\leq k-1}\big(\|\langle t-r \rangle\nabla\Gamma^\alpha v\|^2+
\|\langle t-r \rangle\nabla\Gamma^\alpha G\|^2 \big)\\
&\lesssim E_k+\sum_{|\alpha|\leq k-1}\big(\|| t-r |\nabla\Gamma^\alpha v\|^2+
\|| t-r |\nabla\Gamma^\alpha G\|^2 \big).
\end{split}
\end{equation}
Since
\begin{equation}
\nabla\Gamma^\alpha v= \frac{1}{2}\big[ \nabla\Gamma^\alpha v+\nabla\cdot \Gamma^\alpha G\otimes w\big]+
\frac{1}{2}\big[ \nabla\Gamma^\alpha v-\nabla\cdot \Gamma^\alpha G\otimes w\big],
\end{equation}
and
\begin{equation}
\nabla\cdot\Gamma^\alpha G= \frac{1}{2}\big[ \nabla\Gamma^\alpha v+\nabla\cdot \Gamma^\alpha G\otimes w\big]w-
\frac{1}{2}\big[ \nabla\Gamma^\alpha v-\nabla\cdot \Gamma^\alpha G\otimes w\big]w,
\end{equation}
we see that
\begin{equation}
\begin{split}
&\ \ \ \ |t-r|\big[|\nabla\Gamma^\alpha v|+|\nabla\cdot\Gamma^\alpha G|\big]\\
&\lesssim |t+r|| \nabla\Gamma^\alpha v+\nabla\cdot \Gamma^\alpha G\otimes w|+
|t-r|| \nabla\Gamma^\alpha v-\nabla\cdot \Gamma^\alpha G\otimes w|.
\end{split}
\end{equation}
It follows from \eqref{Equation3} that
\begin{equation}
\begin{split}
\|(t-r)\nabla\Gamma^\alpha v\|^2+\|(t-r)\nabla\cdot\Gamma^\alpha G\|^2\lesssim E_k+\|N_k\|^2.
\end{split}
\end{equation}
Then by Lemma \ref{GequationG}, Lemma \ref{Lemma6.1}, we obtain
 \begin{equation}
 \begin{split}
X_k(t)
&\lesssim E_k+\sum_{|\alpha|\leq k-1}\big(\|| t-r |\nabla\Gamma^\alpha v\|^2+
\|| t-r |\nabla\Gamma^\alpha G\|^2 \big)\\
&\lesssim E_k+\sum_{|\alpha|\leq k-1}\big(\|| t-r |\nabla\Gamma^\alpha v\|^2+
\|| t-r |\nabla\cdot\Gamma^\alpha G\|^2 \big)+Q_\alpha\\
&\lesssim E_k+\|N_k\|^2+Q_\alpha\\
&\lesssim E_k+E_k^2+E_kX_k+X_kE_k^{1/2}.
\end{split}
\end{equation}
Under the assumption that $E_k\ll 1$, we obtain
\begin{equation}
X_k\lesssim E_k.
\end{equation}
\end{proof}
\section{Estimate for Good Unknown}

\begin{lemma}
Suppose that $(v, F)=(v, I+G)$, $(v, G)\in H_\Gamma^k$, solves \eqref{EquationM}.
If
$E_k(t)\ll 1$, then
for any $|\alpha|+4\leq k$, we have the following estimate
\begin{equation}\label{DecayOS2}
\begin{split}
&\ \ \ \ |\partial_r\Gamma^\alpha v+\partial_r\Gamma^\alpha Gw|_{L^\infty(r\geq \frac{\langle t \rangle}{16})}\\
&\lesssim \langle t\rangle^{-\frac{3}{2}}(E_{|\alpha|+3}^{\frac{1}{2}}+E_{|\alpha|+4}).
\end{split}
\end{equation}
\end{lemma}
\begin{proof}
Using \eqref{LastE}, we have
\begin{equation}
\begin{split}
&\ \ \ \ r\partial_r \Gamma^\alpha v+tw\cdot\nabla \Gamma^\alpha v+t\nabla \cdot \Gamma^\alpha G+r\partial_r \Gamma^\alpha G w\\
&=S\Gamma^\alpha v-t f_\alpha +t\nabla \Gamma^\alpha p+(S\Gamma^\alpha G-tg_\alpha)\cdot w,
\end{split}
\end{equation}
hence, using the decomposition of $\nabla$ \eqref{Decompose}, we have
\begin{equation}
\begin{split}
&\ \ \ \ (r+t)(\partial_r \Gamma^\alpha v+\partial_r \Gamma^\alpha Gw)\\&=\frac{t}{r}(w\wedge \Omega)\cdot \Gamma^\alpha G
+S\Gamma^\alpha v-t f_\alpha +t\nabla \Gamma^\alpha p+(S\Gamma^\alpha G-tg_\alpha)\cdot w.
\end{split}
\end{equation}

In the region $r\geq \frac{\langle t \rangle}{16}$, we still need to squeeze out an additional decay factor of $\langle t\rangle^{-1/2}$. Multiplying this inequality with $\langle t\rangle^{\frac{1}{2}}$, as $r\geq \frac{\langle t \rangle}{16}$, using Lemma \ref{Ginequality},
\begin{equation}
\begin{split}
&\ \ \ \ \langle t\rangle^{\frac{1}{2}}(r+t)|\partial_r \Gamma^\alpha v+\partial_r \Gamma^\alpha Gw|\\
&\lesssim\langle t\rangle^{\frac{1}{2}}|(w\wedge \Omega)\cdot \Gamma^\alpha G|
+\langle t\rangle^{\frac{1}{2}}|S\Gamma^\alpha v|+\langle t\rangle^{\frac{1}{2}} |tf_\alpha|
+\langle t\rangle^{\frac{3}{2}}|\nabla \Gamma^\alpha p|\\
&\ \ \ \ +\langle t\rangle^{\frac{1}{2}}|S\Gamma^\alpha G|+\langle t\rangle^{\frac{1}{2}}|tg_\alpha|\\
&\lesssim |\langle t\rangle^{\frac{3}{2}}\nabla \Gamma^\alpha p|+\langle t\rangle^{\frac{1}{2}}L_{|\alpha|+1}+\langle t\rangle^{\frac{1}{2}}N_{|\alpha|+1}\\
&\lesssim E_{|\alpha|+3}^{\frac{1}{2}}+E_{|\alpha|+3}+|\langle t\rangle^{\frac{3}{2}}\nabla \Gamma^\alpha p|.
\end{split}
\end{equation}
By Sobolev embedding $H^2(\mathbb{R}^3)\subset L^\infty(\mathbb{R}^3)$, Lemma \ref{Pressure}, we have
\begin{equation}\label{EstimatP}
\begin{split}
&\ \ \ \ |\langle t\rangle^{\frac{3}{2}}\nabla \Gamma^\alpha p|\lesssim \|\langle t\rangle^{\frac{3}{2}}\nabla \Gamma^{\alpha+2} p\|
\lesssim\sum_{\beta+\gamma=\alpha+2, |\beta|\leq|\gamma|}\langle t\rangle^{\frac{3}{2}}
\|\partial_j\Gamma^\beta v_i\Gamma^\gamma v_j-\partial_j\Gamma^\beta G_{ik}\Gamma^\gamma G_{jk}\|_{L^2}\\
&=\sum_{\beta+\gamma=\alpha+2, |\beta|\leq|\gamma|}\langle t\rangle^{\frac{3}{2}}
\|\partial_j\Gamma^\beta v_i\Gamma^\gamma v_j-\partial_j\Gamma^\beta G_{ik}\Gamma^\gamma G_{jk}\|_{L^2(r\leq \frac{\langle t \rangle}{16})}\\
&\ \ \ \ +\sum_{\beta+\gamma=\alpha+2, |\beta|\leq|\gamma|}\langle t\rangle^{\frac{3}{2}}
\|\partial_j\Gamma^\beta v_i\Gamma^\gamma v_j-\partial_j\Gamma^\beta G_{ik}\Gamma^\gamma G_{jk}\|_{L^2(r\geq \frac{\langle t \rangle}{16})}\\
&=P_1+P_2,
\end{split}
\end{equation}
here and in what follows, the notation $\Gamma^{\alpha+r}$ means $\Gamma^{\alpha+\beta}$ for $|\beta|=r$.
By \eqref{DecayT}, and Lemma \ref{EX}, $P_1$ is estimated as follows
\begin{equation}\label{EstimatP_1}
\begin{split}
P_1&\lesssim \sum_{\beta+\gamma=\alpha+2, |\beta|\leq|\gamma|}\langle t\rangle^{\frac{3}{2}}
\|\Gamma^\gamma U\nabla \Gamma^\beta U\|_{L^2(r\leq \frac{\langle t \rangle}{16})}\\
&\lesssim |\langle t\rangle\Gamma^{\alpha+2} U|_{L^\infty(r\leq \frac{\langle t \rangle}{16})}
\|\langle t-r\rangle\nabla \Gamma^{[\frac{\alpha+2}{2}]}U\|\\
&\lesssim (E_{|\alpha|+2}^{\frac{1}{2}}+X_{|\alpha|+4}^{\frac{1}{2}})X_{[\frac{\alpha+2}{2}]+1}^{\frac{1}{2}}
\lesssim E_{|\alpha|+4}.
\end{split}
\end{equation}
By the decomposition of $\nabla$ \eqref{Decompose}
and Lemma \ref{Decayr3/2}, Lemma \ref{Ginequality},
\begin{equation}\label{EstimatP_2}
\begin{split}
P_2&\lesssim\sum_{\beta+\gamma=\alpha+2, |\beta|\leq|\gamma|}
\|r^{\frac{3}{2}}(w_j\partial_r\Gamma^\beta v_i\Gamma^\gamma v_j-w_j\partial_r\Gamma^\beta G_{ik}\Gamma^\gamma G_{jk})\|_{L^2(r\geq \frac{\langle t \rangle}{16})}\\
&\ \ \ \ +\sum_{\beta+\gamma=\alpha+2, |\beta|\leq|\gamma|}
\|r^{\frac{1}{2}}((w\wedge \Omega)_j\Gamma^\beta v_i\Gamma^\gamma v_j-(w\wedge \Omega)_j\Gamma^\beta G_{ik}\Gamma^\gamma G_{jk})\|_{L^2(r\geq \frac{\langle t \rangle}{16})}\\
&\lesssim E_{|\alpha|+4}.
\end{split}
\end{equation}
Combining all the estimates above, we obtain the estimate \eqref{DecayOS2}.
\end{proof}

\section{Local Energy Decay}
In this section, we will prove that $\langle t\rangle^{3/2}|\nabla \Gamma^\alpha U|_{L^\infty(r\leq \frac{\langle t \rangle}{16})}$
is bounded by the energy, hence, we obtain $(1+t)^{-3/2}$ time decay in $L^\infty(r\leq \frac{\langle t \rangle}{16})$.
Which is important for us to prove our main theorem.
\begin{lemma}\label{Lemma 8.1}
Suppose that $(v, G)$ is a solution of
\begin{equation}\label{Tequation}
\begin{cases}
\partial_t v - \nabla\cdot G = f,\\
\partial_t G  - \nabla v = g,\\
\nabla\cdot v = 0,\quad \nabla\cdot G^T= 0,
\end{cases}
\end{equation}
then inside the cone $r\leq \frac{\langle t\rangle}{8}$, we have the following
\begin{equation}\label{Tcone}
\begin{split}
&\ \ \ \ \|t^{3/2}\nabla^2 G\|_{L^2(r\leq \frac{\langle t\rangle}{8})}^2+\|t^{3/2}\nabla^2  v\|_{L^2(r\leq \frac{\langle t\rangle}{8})}^2
\lesssim\|\eta t^{3/2}\nabla^2 G\|^2+\|\eta t^{3/2}\nabla^2  v\|^2\\
&\lesssim \|\eta t^{3/2} \nabla f\|^2 +\|\eta t^{1/2} \nabla S v\|^2+\|\eta t^{1/2} \nabla^2 v\|^2+\|\eta t^{1/2} \nabla v\|^2
+\|\eta t^{3/2} \nabla g\|^2 \\
&\ \ \ \ +\|\eta t^{1/2} \nabla S G\|^2+\|\eta t^{1/2} \nabla^2 G\|^2+\|t^{1/2} \nabla G\|_{L^2(r\leq \frac{\langle t \rangle}{4})}^2
+\|\eta t^{3/2}\nabla h_0\|^2,
\end{split}
\end{equation}
here $h_0=G_{lj}\partial_l G_{ik}-G_{lk}\partial_l G_{ij}$ and $\eta(t, x)=\varphi(r/\frac{\langle t \rangle}{8})$, $\varphi$ is as before (Lemma \ref{DecayT_0}),
\end{lemma}

\begin{proof}
Multiplying the first and second equation of \eqref{Tequation} by t, moving the $t\partial_t$
terms on the right and rewriting them using the scaling operator gives
\begin{equation}\label{Tequation1}
\begin{cases}
- t\nabla\cdot G = tf-S v +r\partial_r v=tf-S v + x\cdot\nabla v, \\
- t\nabla v = tg-S G +r\partial_r G=tg-S G +x\cdot\nabla  G. \\
\end{cases}
\end{equation}
Taking one more derivative on both side of equation \eqref{Tequation1}, we obtain
\begin{equation}\label{Tequation2}
\begin{cases}
- t\nabla\cdot\nabla G = t\nabla f-\nabla S v +\nabla x\cdot\nabla v+ x\cdot\nabla\nabla v, \\
- t\nabla^2 v = t\nabla g-\nabla S G +\nabla x\cdot\nabla G+x\cdot\nabla\nabla G. \\
\end{cases}
\end{equation}
Equation \eqref{Tequation2} will be our starting point for the derivation of the estimates \eqref{Tcone}.

To begin we multiply equation \eqref{Tequation2} by $\eta t^{1/2}$, and use the triangle
inequality to obtain
\begin{equation}
\begin{split}
&\ \ \ \ \|\eta t^{3/2} \nabla\cdot \nabla G\|^2+\|\eta t^{3/2}\nabla^2  v\|^2\\
&\leq 4\|\eta t^{3/2} \nabla f\|^2 +4\|\eta t^{1/2} \nabla S v\|^2+4\|\eta t^{1/2} \nabla v\|^2+4\|\eta t^{1/2} r\nabla^2 v\|^2\\
&\ \ \ \ +4\|\eta t^{3/2} \nabla g\|^2 +4\|\eta t^{1/2} \nabla S G\|^2+4\|\eta t^{1/2} \nabla G\|^2+4\|\eta t^{1/2} r\nabla^2 G\|^2.\\
\end{split}
\end{equation}
Notice that $r\leq \frac{\langle t\rangle}{4}$ on supp$\eta$, hence we have
\begin{equation}\label{8.7}
\begin{split}
&\ \ \ \ \|\eta t^{3/2} \nabla\cdot \nabla G\|^2+\|\eta t^{3/2}\nabla^2  v\|^2\\
&\leq 4\|\eta t^{3/2} \nabla f\|^2 +4\|\eta t^{1/2} \nabla S v\|^2+4\|\eta t^{1/2} \nabla v\|^2+\frac{1}{4}\|\eta t^{1/2} \nabla^2 v\|^2\\
&\ \ \ \ +4\|\eta t^{3/2} \nabla g\|^2 +4\|\eta t^{1/2} \nabla S G\|^2+4\|\eta t^{1/2} \nabla G\|^2+\frac{1}{4}\|\eta t^{1/2}\nabla^2 G\|^2\\
&\ \ \ \ +\frac{1}{4}\|\eta t^{3/2}\nabla^2 v\|^2+\frac{1}{4}\|\eta t^{3/2}\nabla^2 G\|^2.
\end{split}
\end{equation}

Now we compute
\begin{equation}
\begin{split}
\|\eta t^{3/2} \nabla\cdot \nabla G\|^2&=\int \eta^2 t^3 \partial_k \nabla G_{ik}\partial_j \nabla G_{ij}\\
&=-\int 2\eta \eta^{\prime}t^3\frac{8x_k}{\langle t\rangle r}\nabla G_{ik}\partial_j \nabla G_{ij}
-\eta^2 t^3 \nabla G_{ik}\partial_j \partial_k\nabla G_{ij}\\
&=-\int 2\eta \eta^{\prime}t^3\frac{8x_k}{\langle t\rangle r}\nabla G_{ik}\partial_j \nabla G_{ij}
+2\eta \eta^{\prime}t^3\frac{8x_j}{\langle t\rangle r} \nabla G_{ik}\partial_k\nabla G_{ij}\\
&\ \ \ \ +\eta^2 t^3 \partial_j\nabla G_{ik} \partial_k\nabla G_{ij}\\
&=-\int 2\eta \eta^{\prime}t^3\frac{8x_k}{\langle t\rangle r}\nabla G_{ik}\partial_j \nabla G_{ij}
+2\eta \eta^{\prime}t^3\frac{8x_j}{\langle t\rangle r} \nabla G_{ik}\partial_k\nabla G_{ij}\\
&\ \ \ \ +\eta^2 t^3 \partial_j\nabla G_{ik}\nabla h_0+
\|\eta t^{3/2} \nabla^2 G\|^2.
\end{split}
\end{equation}
The last equality follows from \eqref{ConstraintG0}, hence
\begin{equation}
\begin{split}
&\ \ \ \ \|\eta t^{3/2} \nabla^2G\|^2\\
&\leq\|\eta t^{3/2} \nabla\cdot \nabla G\|^2+16\text{max}|\eta^{\prime}|\|\eta t^{3/2} \nabla\cdot\nabla G\|\|t^{1/2}\nabla G\|_{L^2(r\leq \frac{\langle t \rangle}{4})}\\
&\ \ \ \ +16\text{max}|\eta^{\prime}|\|\eta t^{3/2} \nabla^2 G\|\|t^{1/2}\nabla G\|_{L^2(r\leq \frac{\langle t \rangle}{4})}
+\|\eta t^{3/2} \nabla^2 G\|\|\eta t^{3/2} \nabla h_0\|\\
&\leq \frac{3}{2}\|\eta t^{3/2} \nabla\cdot\nabla G\|^2+\frac{1}{2}\|\eta t^{3/2} \nabla^2G\|^2+C(\|t^{1/2}\nabla G\|_{L^2(r\leq \frac{\langle t \rangle}{4})}^2+\|\eta t^{3/2} \nabla h_0\|^2),
\end{split}
\end{equation}
we obtain
\begin{equation}
\begin{split}
\|\eta t^{3/2} \nabla^2G\|^2
\leq 3\|\eta t^{3/2} \nabla\cdot\nabla G\|^2+C(\|t^{1/2}\nabla G\|_{L^2(r\leq \frac{\langle t \rangle}{4})}^2+\|\eta t^{3/2} \nabla h_0\|^2).
\end{split}
\end{equation}
Combining with \eqref{8.7}, we obtain
\begin{equation}
\begin{split}
&\ \ \ \ \|\eta t^{3/2}\nabla^2 G\|^2+\|\eta t^{3/2}\nabla^2  v\|^2\\
&\lesssim \|\eta t^{3/2} \nabla f\|^2 +\|\eta t^{1/2} \nabla S v\|^2+\|\eta t^{1/2} \nabla^2 v\|^2+\|\eta t^{1/2} \nabla v\|^2
+\|\eta t^{3/2} \nabla g\|^2 \\
&\ \ \ \ +\|\eta t^{1/2} \nabla S G\|^2+\|\eta t^{1/2} \nabla^2 G\|^2+\|t^{1/2} \nabla G\|_{L^2(r\leq \frac{\langle t \rangle}{4})}^2
+\|\eta t^{3/2}\nabla h_0\|^2.
\end{split}
\end{equation}
Then Lemma \ref{Lemma 8.1} follows.
\end{proof}
Let
\begin{equation}
\Psi_l(U)=\sum_{|\alpha|\leq l}\|\eta\langle t\rangle^{3/2}\nabla^2 \Gamma^\alpha U\|.
\end{equation}
\begin{lemma}\label{Genergy}
Suppose that $(v, F)=(v, I+G)$, $(v, G)\in H_\Gamma^k$, solves \eqref{EquationM}.
If
$E_k(t)\ll 1$, then for any $l+4\leq k$, we have $\Psi_l(U)\lesssim E_{l+4}^{\frac{1}{2}}$.
\end{lemma}
\begin{proof}
Combining Lemma \ref{Tequation} with equations \eqref{Equation} give us
\begin{equation}
\begin{split}
\Psi_l^2(U)&\lesssim\sum_{|\alpha|\leq l}\|\eta\langle t\rangle^{3/2}\nabla^2 \Gamma^\alpha U\|^2\\
&\lesssim E_{l+2}+\|t^{3/2}\nabla\Gamma^{\alpha+1}p\|^2+
\|\eta t^{3/2} \nabla f_\alpha\|^2 +\|\eta t^{1/2} \nabla S \Gamma^\alpha v\|^2\\
&\ \ \ \ +\|\eta t^{1/2} \nabla^2 \Gamma^\alpha v\|^2+\|\eta t^{1/2} \nabla \Gamma^\alpha v\|^2
+\|\eta t^{3/2} \nabla g_\alpha\|^2 +\|\eta t^{1/2} \nabla S \Gamma^\alpha G\|^2\\
&\ \ \ \ +\|\eta t^{1/2} \nabla^2 \Gamma^\alpha G\|^2+\|t^{1/2} \nabla \Gamma^\alpha G\|_{L^2(r\leq \frac{\langle t \rangle}{4})}^2
+\|\eta t^{3/2}\nabla h_\alpha\|^2.
\end{split}
\end{equation}
In a similar fashion as we estimate \eqref{EstimatP}-\eqref{EstimatP_2} earlier, we have
\begin{equation}
\begin{split}
t^{3/2}\|\nabla\Gamma^{\alpha+1}p\|\lesssim E_{l+4}.
\end{split}
\end{equation}
On the other hand, by the definition of $f_\alpha$, $g_\alpha$, $h_\alpha$, inequality \eqref{DecayT} and Lemma \ref{EX}, we have
\begin{equation}
\begin{split}
&\ \ \ \ \|\eta t^{3/2} \nabla f_\alpha\|^2+\|\eta t^{3/2} \nabla g_\alpha\|^2
+\|\eta t^{3/2}\nabla h_\alpha\|^2\\
&\lesssim E_{l+4}^2,
\end{split}
\end{equation}
hence, we have
\begin{equation}
\Psi_l^2(U)\lesssim E_{l+2}+E_{l+4}^2+ X_{l+2}\lesssim E_{l+4}.
\end{equation}
\end{proof}

\begin{rem}
By Lemma \ref{DecayT_0}, replacing $f$ by $\langle t\rangle^{1/2}\nabla f$, we have
\begin{equation}
\begin{split}
&\ \ \ \ \langle t\rangle^{3/2}|\nabla f|_{L^\infty(r\leq \frac{\langle t \rangle}{16})}\\
&\lesssim \langle t\rangle^{1/2}\|\nabla f\|_{L^2(r\leq \frac{\langle t \rangle}{8})}+\langle t\rangle^{3/2}\|\nabla^2 f\|_{L^2(r\leq \frac{\langle t \rangle}{8})}
+\langle t\rangle^{3/2}\|\nabla^3 f\|_{L^2(r\leq \frac{\langle t \rangle}{8})}\\
&\lesssim X_1^{\frac{1}{2}}+\langle t\rangle^{3/2}\|\nabla^2 f\|_{L^2(r\leq \frac{\langle t \rangle}{8})}+\langle t\rangle^{3/2}\|\nabla^3 f\|_{L^2(r\leq \frac{\langle t \rangle}{8})},
\end{split}
\end{equation}
then, for any $\alpha$, we have
\begin{equation}\label{Decayt3/2}
\begin{split}
&\ \ \ \  \langle t\rangle^{3/2}|\nabla \Gamma^\alpha f|_{L^\infty(r\leq \frac{\langle t \rangle}{16})}\\
&\lesssim X_{|\alpha|+1}^{\frac{1}{2}}+\langle t\rangle^{3/2}\|\nabla^2 \Gamma^\alpha f\|_{L^2(r\leq \frac{\langle t \rangle}{8})}+\langle t\rangle^{3/2}\|\nabla^3 \Gamma^\alpha f\|_{L^2(r\leq \frac{\langle t \rangle}{8})}.
\end{split}
\end{equation}
Hence, inside the cone, under the condition of
Lemma \ref{Genergy}, for any $|\alpha|+5\leq k$, we have
\begin{equation}\label{DecayTT3/2}
\begin{split}
\langle t\rangle^{3/2}|\nabla \Gamma^\alpha U|_{L^\infty(r\leq \frac{\langle t \rangle}{16})}
\lesssim E_{|\alpha|+5}^{\frac{1}{2}}.
\end{split}
\end{equation}
\end{rem}

\section{Energy Estimate with a Ghost Weight}
Choosing $q = q(t - r)$, with $q(\sigma) = \int_0^\sigma\frac{1}{1
+ z^2}dz$ so that $q^\prime(\sigma) = \frac{1}{1 + \sigma^2}$ and
$|q(\sigma)| \leq \frac{\pi}{2}$. Let $|\alpha| \leq k$, taking
the inner product of the first and second equation in \eqref{Equation}
with $e^{- q}\Gamma^\alpha v$ and $e^{- q}\Gamma^\alpha G$
respectively, and then adding them up, one has
\begin{equation}
\begin{split}
&\ \ \ \ \int\Big(e^{- q}\partial_t(|\Gamma^\alpha v|^2 + |\Gamma^\alpha
  G|^2) - 2e^{- q}(\Gamma^\alpha v_i\partial_j\Gamma^\alpha G_{ij}
  + \Gamma^\alpha G_{ij}\partial_j\Gamma^\alpha v_i)\Big)dx\\
&= - 2\int e^{- q}\Gamma^\alpha v\cdot\nabla\Gamma^\alpha pdx +
  2\int e^{- q}(f_\alpha\cdot \Gamma^\alpha v + (g_\alpha)_{ij} \Gamma^\alpha G_{ij})dx.
\end{split}
\end{equation}
Integration by parts gives that
\begin{equation}
\begin{split}
&\ \ \ \ \frac{d}{dt}\int e^{- q}(|\Gamma^\alpha v|^2 + |\Gamma^\alpha
  G|^2)dx\\
&= - \int e^{- q}\big[\partial_tq(|\Gamma^\alpha v|^2 + |\Gamma^\alpha G|^2)
- 2\partial_jq\Gamma^\alpha v_i\Gamma^\alpha G_{ij}\big]dx\\
&\ \ \ \  -\ 2\int e^{- q}\Gamma^\alpha v\cdot\nabla\Gamma^\alpha
pdx + 2\int e^{- q}(f_\alpha\cdot \Gamma^\alpha v + (g_\alpha)_{ij} \Gamma^\alpha
  G_{ij})dx\\
&= - \int \frac{e^{- q}}{1 + (t - r)^2}\big[(|\Gamma^\alpha v|^2 + |\Gamma^\alpha G|^2)
+2w_j\Gamma^\alpha v_i\Gamma^\alpha G_{ij}\big]dx\\
&\ \ \ \  -\ 2\int e^{- q}\Gamma^\alpha v\cdot\nabla\Gamma^\alpha
pdx + 2\int e^{- q}(f_\alpha\cdot \Gamma^\alpha v + (g_\alpha)_{ij} \Gamma^\alpha G_{ij})dx\\
&= - \int \frac{e^{- q}}{1 + (t - r)^2}\big(|\Gamma^\alpha v +
\Gamma^\alpha G\omega|^2 + |\Gamma^\alpha G-\Gamma^\alpha G w\otimes w|^2\big)dx\\
&\ \ \ \ - 2\int e^{- q}\Gamma^\alpha v\cdot\nabla\Gamma^\alpha
pdx + 2\int e^{-q}\big[f_\alpha\cdot \Gamma^\alpha v
+ (g_\alpha)_{ij} \Gamma^\alpha G_{ij}\big]dx.
\end{split}
\end{equation}
That is
\begin{equation}\label{Gequation}
\begin{split}
&\frac{d}{dt}\int e^{- q}(|\Gamma^\alpha v|^2 + |\Gamma^\alpha
  G|^2)dx\\
&\ \ \ \ + \int \frac{e^{- q}}{1 + (t - r)^2}\big(|\Gamma^\alpha v +
\Gamma^\alpha G\omega|^2 + |\Gamma^\alpha G-\Gamma^\alpha Gw\otimes w|^2\big)dx\\
&= 2\int e^{-q}\big[f_\alpha\cdot \Gamma^\alpha v
  + (g_\alpha)_{ij} \Gamma^\alpha G_{ij}\big]dx- 2\int e^{- q}\Gamma^\alpha v\cdot\nabla\Gamma^\alpha
  pdx.
\end{split}
\end{equation}
We emphasize that here we do not use integration by parts in the term involving pressure. We
also point out that we cannot use the approach of Lemma \ref{Lemma6.1} to estimate the nonlinear terms
because we now have that $|\alpha|\leq k$ rather than $|\alpha|\leq k-1$ as we had earlier.
To simplify the notation a bit, we shall write
\begin{equation}
z=\sum_{|\alpha|\leq k}\big(|\Gamma^\alpha v +
\Gamma^\alpha G\omega|^2 + |\Gamma^\alpha G-\Gamma^\alpha Gw\otimes w|^2\big),
\end{equation}
and
\begin{equation}
\begin{split}
Z&= \int \frac{e^{- q}}{1 + (t - r)^2} zdx\\
&=
\sum_{|\alpha|\leq k} \int \frac{e^{- q}}{1 + (t - r)^2}\big(|\Gamma^\alpha v +
\Gamma^\alpha G\omega|^2 + |\Gamma^\alpha G-\Gamma^\alpha Gw\otimes w|^2\big)dx.
\end{split}
\end{equation}

Let us first treat the first term in \eqref{Gequation}. Recall that
$f_\alpha$ and $g_\alpha$ are given by \eqref{Equationfg}. By
divergence-free of $v$ and $G^T$, one has
\begin{equation}
\begin{split}
&\ \ \ \ \int e^{-q}\big[f_\alpha\cdot \Gamma^\alpha v + (g_\alpha)_{ij}
  \Gamma^\alpha G_{ij}\big]dx\\
& = \sum_{\beta + \gamma = \alpha, \gamma \neq \alpha}
  \int e^{-q}\Gamma^\alpha v_i
  \big[\partial_j\Gamma^\gamma G_{ik} \Gamma^\beta G_{jk} -
  \Gamma^\beta v_j\partial_j\Gamma^\gamma v_i\big]dx\\
&\ \ \ \  +\ \sum_{\beta + \gamma = \alpha, \gamma \neq \alpha}\int
  e^{-q}\Gamma^\alpha G_{ik}\big[\partial_j\Gamma^\gamma v_{i} \Gamma^\beta G_{jk}
  - \Gamma^\beta v_j\partial_j\Gamma^\gamma
  G_{ik}\big]dx\\
&\ \ \ \ +\ \frac{1}{2}\int e^{-q}\partial_j
  \big[2\Gamma^\alpha v_i\Gamma^\alpha G_{ik} G_{jk} -
  v_j(|\Gamma^\alpha v|^2 + |\Gamma^\alpha G|^2)\big]dx\\
&=I+ II +III.
\end{split}
\end{equation}
To estimate the last term in this equality, by integration by parts,
we have
\begin{equation}
\begin{split}
2III&=\int e^{-q}\partial_j
  \big[2\Gamma^\alpha v_i\Gamma^\alpha G_{ik} G_{jk} -
  v_j(|\Gamma^\alpha v|^2 + |\Gamma^\alpha G|^2)\big]dx\\
&= - \int \frac{e^{-q}}{1 + (t - r)^2}\big[2\Gamma^\alpha
  v_i\Gamma^\alpha G_{ik} G_{jk}\omega_j -
  (|\Gamma^\alpha v|^2 + |\Gamma^\alpha
  G|^2)v_j\omega_j\big]dx.
\end{split}
\end{equation}
Inside the cone ($|x|\leq\frac{\langle t\rangle}{16}$), this term is bounded by
\begin{equation}
2III\lesssim \frac{1}{(1+t)^2}E_k^{3/2}.
\end{equation}
Away from the origin ($|x|\geq\frac{\langle t\rangle}{16}$), by Lemma \ref{Decayr3/2}, this term is bounded by
\begin{equation}
\begin{split}
2III
&\lesssim\frac{1}{(1+t)^{3/2}}E_kE_2^{1/2}\lesssim\frac{1}{(1+t)^{3/2}}E_k^{3/2}.
\end{split}
\end{equation}
Next, we are going to estimate the terms I and II.

\subsection{Inside the cone}
On the region:  $|x|\leq\frac{\langle t\rangle}{16}$. The terms $I$ and $II$ are bounded by
\begin{equation}
\begin{split}
&\sum_{\substack{\beta+\gamma=\alpha, |\alpha|\leq k\\ \gamma\neq\alpha}}
  \int_{r\leq\frac{\langle t\rangle}{16}} e^{-q}\Gamma^\alpha v_i
  \big[\partial_j\Gamma^\gamma G_{ik} \Gamma^\beta G_{jk} -
  \Gamma^\beta v_j\partial_j\Gamma^\gamma v_i\big]dx\\
&\ \ \ \ + \sum_{\substack{\beta+\gamma=\alpha, |\alpha|\leq k\\ \gamma\neq\alpha}} \int_{r\leq\frac{\langle t\rangle}{16}}
  e^{-q}\Gamma^\alpha G_{ik}\big[\partial_j\Gamma^\gamma v_{i} \Gamma^\beta G_{jk}
  - \Gamma^\beta v_j\partial_j\Gamma^\gamma
  G_{ik}\big]dx\\
&\lesssim\sum_{\substack{\beta+\gamma=\alpha, |\alpha|\leq k\\ \gamma\neq\alpha}} \int_{r\leq\frac{\langle t\rangle}{16}}
|\Gamma^\alpha U \Gamma^\beta U\nabla\Gamma^\gamma U|.
\end{split}
\end{equation}
We make use of the fact that since $\beta+\gamma=\alpha$, $|\alpha|\leq k$ , $\gamma\neq \alpha$,
either $|\beta|\leq k^{\prime}$ or $|\gamma| \leq k^{\prime}$, with $k^{\prime}=[\frac{k}{2}]$. Note that $k\geq 9$, hence,
we have $k^\prime+ 5\leq k$.

In the first case, using \eqref{DecayT} we have:
\begin{equation}
\begin{split}
&\ \ \ \ \sum_{\substack{\beta+\gamma=\alpha, |\alpha|\leq k\\ \gamma\neq\alpha}} \int_{r\leq\frac{\langle t\rangle}{16}}
|\Gamma^\alpha U \Gamma^\beta U\nabla\Gamma^\gamma U|\\
&\lesssim\frac{1}{\langle t\rangle^2} \|\Gamma^k U\|\|\langle t-r\rangle\nabla\Gamma^{k-1}U\|_{L^2({r\leq\frac{\langle t\rangle}{16}})}
|\langle t\rangle\Gamma^{k^{\prime}}U|_{L^\infty({r\leq\frac{\langle t\rangle}{16}})}\\
&\lesssim\frac{1}{\langle t\rangle^2} E_k^{\frac{1}{2}} X_k^{\frac{1}{2}}(E_{k^\prime}^{\frac{1}{2}}+X_{k^\prime+2}^{\frac{1}{2}})
\lesssim\frac{1}{\langle t\rangle^2} E_k^{\frac{3}{2}}.
\end{split}
\end{equation}

In the second case, using \eqref{DecayTT3/2}:
\begin{equation}
\begin{split}
\sum_{\substack{\beta+\gamma=\alpha, |\alpha|\leq k\\ \gamma\neq\alpha}} \int_{r\leq\frac{\langle t\rangle}{16}}
|\Gamma^\alpha U \Gamma^\beta U\partial\Gamma^\gamma U|
&\lesssim\frac{1}{\langle t\rangle^{3/2}}\|\Gamma^k U\|\|\Gamma^k U\||\langle t\rangle^{3/2}\nabla \Gamma^{k^{\prime}}U|_{L^\infty({r\leq\frac{\langle t\rangle}{16}})}\\
&\lesssim \frac{1}{\langle t\rangle^{3/2}}E_k^{\frac{1}{2}}E_k^{\frac{1}{2}}E_{k^\prime+5}^{\frac{1}{2}}
\lesssim\frac{1}{\langle t\rangle^{3/2}} E_k^{\frac{3}{2}}.
\end{split}
\end{equation}
\subsection{Away from the origin}
On the region:  $|x|\geq\frac{\langle t\rangle}{16}$. By the decomposition of $\nabla$
\eqref{Decompose}, the terms $I$ and $II$ can be written as

\begin{equation}
\begin{split}
&\sum_{\substack{\beta+\gamma=\alpha, |\alpha|\leq k\\ \gamma\neq\alpha}}
  \big |\int_{r\geq\frac{\langle t\rangle}{16}} e^{-q}\Gamma^\alpha v_i
  \big[\partial_j\Gamma^\gamma G_{ik} \Gamma^\beta G_{jk} -
  \Gamma^\beta v_j\partial_j\Gamma^\gamma v_i\big]dx\\
&\ \ \ \  + \int_{r\geq\frac{\langle t\rangle}{16}}
  e^{-q}\Gamma^\alpha G_{ik}\big[\partial_j\Gamma^\gamma v_{i} \Gamma^\beta G_{jk}
  - \Gamma^\beta v_j\partial_j\Gamma^\gamma
  G_{ik}\big]dx\big |\\
&=\sum_{\substack{\beta+\gamma=\alpha, |\alpha|\leq k\\ \gamma\neq\alpha}}
  \big |\int_{r\geq\frac{\langle t\rangle}{16}} e^{-q}\Gamma^\alpha v_i
  \big[w_j\partial_r\Gamma^\gamma G_{ik} \Gamma^\beta G_{jk} -
  \Gamma^\beta v_jw_j\partial_r\Gamma^\gamma v_i\big]dx\\
&\ \ \ \  + \int_{r\geq\frac{\langle t\rangle}{16}}
  e^{-q}\Gamma^\alpha G_{ik}\big[w_j\partial_r\Gamma^\gamma v_{i} \Gamma^\beta G_{jk}
  - \Gamma^\beta v_jw_j\partial_r\Gamma^\gamma
  G_{ik}\big]dx\big |\\
&\ \ \ \ + \int_{r\geq\frac{\langle t\rangle}{16}} R_\alpha dx,
\end{split}
\end{equation}
here
\begin{equation}
|R_\alpha|\lesssim \frac{1}{r}\sum_{\substack{\beta+\gamma=\alpha, |\alpha|\leq k\\ \gamma\neq\alpha}}
|(\Gamma^\alpha v, \Gamma^\alpha G)||(\Gamma^\beta v, \Gamma^\beta G)||(\Omega\Gamma^\gamma v, \Omega\Gamma^\gamma G)|,
\end{equation}
since $\beta+\gamma=\alpha, \gamma\neq\alpha, |\alpha|\leq k$,
using Lemma \ref{Ginequality}, the last term is bounded by
\begin{equation}
\begin{split}
\int_{r\geq\frac{\langle t\rangle}{16}} R_\alpha dx\lesssim \frac{1}{\langle t\rangle^2}
\|\Gamma^k U\|^2|r\Gamma^{[\frac{k}{2}]+1}U|_{L^\infty}
\lesssim\frac{1}{\langle t\rangle^2} E_k^{\frac{3}{2}}.
\end{split}
\end{equation}
On the other hand, we have the following simply calculation,
\begin{equation}
\begin{split}
&\ \ \ \ \sum_{\substack{\beta+\gamma=\alpha, |\alpha|\leq k\\ \gamma\neq\alpha}}\int_{r\geq\frac{\langle t\rangle}{16}}
e^{-q}\Gamma^\alpha v_i
  \big[w_j\partial_r\Gamma^\gamma G_{ik} \Gamma^\beta G_{jk} -
  \Gamma^\beta v_jw_j\partial_r\Gamma^\gamma v_i\big]\\
&=-\sum_{\substack{\beta+\gamma=\alpha, |\alpha|\leq k\\ \gamma\neq\alpha}}\int_{r\geq\frac{\langle t\rangle}{16}}
e^{-q}\Gamma^\alpha v_i
  (\Gamma^\beta v_j+\Gamma^\beta G_{jl}w_l)w_j\partial_r\Gamma^\gamma v_i\\
&\ \ \ \ +\sum_{\substack{\beta+\gamma=\alpha, |\alpha|\leq k\\ \gamma\neq\alpha}}\int_{r\geq\frac{\langle t\rangle}{16}}
e^{-q}\Gamma^\alpha v_i
  \Gamma^\beta G_{jl}w_l(w_j\partial_r\Gamma^\gamma v_i+w_j\partial_r\Gamma^\gamma G_{ik}w_k)\\
&\ \ \ \  +\sum_{\substack{\beta+\gamma=\alpha, |\alpha|\leq k\\ \gamma\neq\alpha}}\int_{r\geq\frac{\langle t\rangle}{16}}
e^{-q}\Gamma^\alpha v_i
(\partial_r\Gamma^\gamma G(I-w \otimes w))_{il}w_j (\Gamma^\beta G(I-w\otimes w))_{jl}\\
&=I_1+I_2+I_3.
\end{split}
\end{equation}
And
\begin{equation}
\begin{split}
&\ \ \ \ \sum_{\substack{\beta+\gamma=\alpha, |\alpha|\leq k\\ \gamma\neq\alpha}}\int_{r\geq\frac{\langle t\rangle}{16}}
e^{-q}
\Gamma^\alpha G_{ik}\big[w_j\partial_r\Gamma^\gamma v_{i} \Gamma^\beta G_{jk}
  - \Gamma^\beta v_jw_j\partial_r\Gamma^\gamma
  G_{ik}\big]\\
&=\sum_{\substack{\beta+\gamma=\alpha, |\alpha|\leq k\\ \gamma\neq\alpha}}\int_{r\geq\frac{\langle t\rangle}{16}}
e^{-q}
\Gamma^\alpha G_{ik}
(w_j\partial_r\Gamma^\gamma v_{i}+w_j\partial_r\Gamma^\gamma G_{il}w_l) \Gamma^\beta G_{jk}\\
&\ \ \ \ - \sum_{\substack{\beta+\gamma=\alpha, |\alpha|\leq k\\ \gamma\neq\alpha}}\int_{r\geq\frac{\langle t\rangle}{16}}
e^{-q}
\Gamma^\alpha G_{ik}
(\Gamma^\beta v_j+\Gamma^\beta G_{jl}w_l)w_j\partial_r\Gamma^\gamma G_{ik}\\
&\ \ \ \ +\sum_{\substack{\beta+\gamma=\alpha, |\alpha|\leq k\\ \gamma\neq\alpha}}\int_{r\geq\frac{\langle t\rangle}{16}}
e^{-q}
\Gamma^\alpha G_{ik}
w_j\Gamma^\beta G_{jl}w_l(\partial_r\Gamma^\gamma G(I-w\otimes w))_{ik}\\
&\ \ \ \ -\sum_{\substack{\beta+\gamma=\alpha, |\alpha|\leq k\\ \gamma\neq\alpha}}\int_{r\geq\frac{\langle t\rangle}{16}}
e^{-q}
\Gamma^\alpha G_{ik}
w_j\partial_r\Gamma^\gamma G_{il}w_l (\Gamma^\beta G(I-w\otimes w))_{jk}\\
&=II_1+II_2+II_3+II_4.
\end{split}
\end{equation}

In the following, we will estimate the above terms in the region $r\geq\frac{\langle t\rangle}{16}$ one by one.
We make use of the fact that since $\beta+\gamma=\alpha$, $|\alpha|\leq k$ and $k\geq 9$,
either $|\beta|\leq k^{\prime}$ or $|\gamma| \leq k^{\prime}$, with $k^{\prime}=[\frac{k}{2}]$. Hence, we have

\begin{equation}
\begin{split}
I_1&=-\sum_{\substack{\beta+\gamma=\alpha, |\alpha|\leq k\\ \gamma\neq\alpha}} \int_{r\geq\frac{\langle t\rangle}{16}}
e^{-q}\Gamma^\alpha v_i(\Gamma^\beta v_j+\Gamma^\beta G_{jl}w_l)w_j\partial_r\Gamma^\gamma v_i\\
&\lesssim\sum_{\substack{\beta+\gamma=\alpha, |\alpha|\leq k\\ \gamma\neq\alpha}} \int_{r\geq\frac{\langle t\rangle}{16}}
|\Gamma^\alpha v||\Gamma^\beta vw+w\Gamma^\beta G w||\partial_r\Gamma^\gamma v|\\
&\lesssim \int_{r\geq\frac{\langle t\rangle}{16}}\frac{1}{r^{3/2}}|\Gamma^kv|r^{3/2}|\Gamma^{k^{\prime}} vw+w\Gamma^{k^{\prime}} G w|
|\partial_r\Gamma^{k-1} U|\\
&\ \ \ \ +\int_{r\geq\frac{\langle t\rangle}{16}}|\Gamma^kv|\frac{1}{\langle r\rangle}\frac{1}{\langle t-r\rangle}|\Gamma^k v+\Gamma^k G w|
\langle r\rangle\langle t-r\rangle|\partial_r\Gamma^{k^{\prime}} U|,
\end{split}
\end{equation}
the first term of this inequality is estimated by Lemma \ref{Decayr3/2}, and the second term is estimated by Lemma \ref{Ginequality2},
\begin{equation}
\begin{split}
I_1&\lesssim \frac{1}{(1+t)^{3/2}}E_k^{\frac{1}{2}}E_{k^{\prime}+2}^{\frac{1}{2}}E_k^{\frac{1}{2}}
+\frac{1}{1+t}E_k^{\frac{1}{2}}Z^{\frac{1}{2}}X_{k^{\prime}+3}^{\frac{1}{2}}\\
&\lesssim \frac{1}{(1+t)^{3/2}}E_k^{\frac{3}{2}}+\frac{1}{1+t}E_kZ^{\frac{1}{2}}.
\end{split}
\end{equation}
Similarly, the terms $I_2$ can be bounded as follows
\begin{equation}
\begin{split}
I_2&=\sum_{\substack{\beta+\gamma=\alpha, |\alpha|\leq k\\ \gamma\neq\alpha}} \int_{r\geq\frac{\langle t\rangle}{16}}
e^{-q}\Gamma^\alpha v_i\Gamma^\beta G_{jl}w_l(w_j\partial_r\Gamma^\gamma v_i+w_j\partial_r\Gamma^\gamma G_{ik}w_k)\\
&\lesssim\sum_{\substack{\beta+\gamma=\alpha, |\alpha|\leq k\\ \gamma\neq\alpha}} \int_{r\geq\frac{\langle t\rangle}{16}}
|\Gamma^\alpha v||w\Gamma^\beta G w||\partial_r\Gamma^\gamma v+\partial_r\Gamma^\gamma Gw|\\
&\lesssim \int_{r\geq\frac{\langle t\rangle}{16}}
|\Gamma^kv||w\Gamma^{k^{\prime}} G w||\partial_r\Gamma^{k-1} v+\partial_r\Gamma^{k-1} Gw|\\
&\ \ \ \ +\int_{r\geq\frac{\langle t\rangle}{16}}
|\Gamma^kv||w\Gamma^k G w||\partial_r\Gamma^{k^{\prime}} v+\partial_r\Gamma^{k^{\prime}} Gw|.
\end{split}
\end{equation}
The first term is handled using Lemma \ref{Decayr3/2}, the second term is handled using \eqref{DecayOS2},
thus we have
\begin{equation}
\begin{split}
II_1&\lesssim \frac{1}{(1+t)^{3/2}}E_k^{\frac{1}{2}}E_{k^{\prime}+2}^{\frac{1}{2}}E_k^{\frac{1}{2}}
+\frac{1}{\langle t\rangle^{3/2}} E_k(E_{k^\prime+3}^{\frac{1}{2}}+E_{k^\prime+4})\\
&\lesssim \frac{1}{(1+t)^{3/2}}E_k^{\frac{3}{2}}.
\end{split}
\end{equation}

The terms $I_3$ are handled in a similar fashion as $I_1$, $I_2$
\begin{equation}
\begin{split}
I_3&=\sum_{\substack{\beta+\gamma=\alpha, |\alpha|\leq k\\ \gamma\neq\alpha}} \int_{r\geq\frac{\langle t\rangle}{16}}
e^{-q}\Gamma^\alpha v_i(\partial_r\Gamma^\gamma G(I-w \otimes w))_{il} w_j(\Gamma^\beta G(I-w\otimes w))_{jl}\\
&\lesssim\sum_{\substack{\beta+\gamma=\alpha, |\alpha|\leq k\\ \gamma\neq\alpha}} \int_{r\geq\frac{\langle t\rangle}{16}}
|\Gamma^\alpha v||\partial_r\Gamma^\gamma G(I-w\otimes w)||w\Gamma^\beta G(I-w\otimes w)|\\
&\lesssim \int_{r\geq\frac{\langle t\rangle}{16}}|\Gamma^k v||\partial_r\Gamma^{k-1} G(I-w\otimes w)||w\Gamma^{k^{\prime}} G(I-w\otimes w)|\\
&\ \ \ \ + \int_{r\geq\frac{\langle t\rangle}{16}}|\Gamma^k v||\partial_r\Gamma^{k^{\prime}} G(I-w\otimes w)||w\Gamma^k G(I-w\otimes w)|\\
&\lesssim \frac{1}{(1+t)^{3/2}}E_k^{\frac{1}{2}}E_k^{\frac{1}{2}}E_{k^{\prime}+2}^{\frac{1}{2}}
+\frac{1}{1+t}E_k^{\frac{1}{2}}X_{k^{\prime}+3}^{\frac{1}{2}}Z^{\frac{1}{2}}\\
&\lesssim \frac{1}{(1+t)^{3/2}}E_k^{\frac{3}{2}}+\frac{1}{1+t}E_kZ^{\frac{1}{2}}.
\end{split}
\end{equation}

It remains to estimate the four terms of $II$.
In a similar fashion, the first term of $II$ is bounded by
\begin{equation}
\begin{split}
II_1&=\sum_{\substack{\beta+\gamma=\alpha, |\alpha|\leq k\\ \gamma\neq\alpha}} \int_{r\geq\frac{\langle t\rangle}{16}}
e^{-q}\Gamma^\alpha G_{ik}(w_j\partial_r\Gamma^\gamma v_{i}+w_j\partial_r\Gamma^\gamma G_{il}w_l) \Gamma^\beta G_{jk}\\
&\lesssim\sum_{\substack{\beta+\gamma=\alpha, |\alpha|\leq k\\ \gamma\neq\alpha}} \int_{r\geq\frac{\langle t\rangle}{16}}
|\Gamma^\alpha G||\partial_r\Gamma^\gamma v+\partial_r\Gamma^\gamma Gw||w\Gamma^\beta G|\\
&\lesssim \int_{r\geq\frac{\langle t\rangle}{16}}|\Gamma^k G|r^{-3/2}(\partial_r\Gamma^{k-1} v+\partial_r\Gamma^{k-1} Gw)
(r^{3/2}w\Gamma^{k^\prime} G)\\
&\ \ \ \ +\int_{r\geq\frac{\langle t\rangle}{16}}|\Gamma^k G|(\partial_r\Gamma^{k^\prime} v+\partial_r\Gamma^{k^\prime} Gw)
w\Gamma^{k} G.
\end{split}
\end{equation}
The first term is handled using Lemma \ref{Decayr3/2}, the second term is handled using \eqref{DecayOS2},
thus we have
\begin{equation}
\begin{split}
II_1&\lesssim \frac{1}{(1+t)^{3/2}}E_k^{\frac{1}{2}}E_k^{\frac{1}{2}}E_{k^{\prime}+2}^{\frac{1}{2}}
+\frac{1}{\langle t\rangle^{3/2}} E_k(E_{k^\prime+3}^{\frac{1}{2}}+E_{k^\prime+4})\\
&\lesssim \frac{1}{(1+t)^{3/2}}E_k^{\frac{3}{2}}.
\end{split}
\end{equation}
The second term of $II$ is handled in a similar fashion as $I_1$:
\begin{equation}
\begin{split}
II_2&=-\sum_{\substack{\beta+\gamma=\alpha, |\alpha|\leq k\\ \gamma\neq\alpha}} \int_{r\geq\frac{\langle t\rangle}{16}}
e^{-q}\Gamma^\alpha G_{ik}(\Gamma^\beta v_j+\Gamma^\beta G_{jl}w_l)w_j\partial_r\Gamma^\gamma G_{ik}\\
&\lesssim\sum_{\substack{\beta+\gamma=\alpha, |\alpha|\leq k\\ \gamma\neq\alpha}} \int_{r\geq\frac{\langle t\rangle}{16}}
|\Gamma^\alpha G||\Gamma^\beta vw+w\Gamma^\beta G w||\partial_r\Gamma^\gamma G|\\
&\lesssim \int_{r\geq\frac{\langle t\rangle}{16}}|\Gamma^k G|\frac{1}{r^{3/2}}r^{3/2}|\Gamma^{k^{\prime}} vw+w\Gamma^{k^{\prime}} G w|
|\partial_r\Gamma^{k-1} G|\\
&\ \ \ \ +\int_{r\geq\frac{\langle t\rangle}{16}}|\Gamma^k G|\frac{1}{\langle r\rangle}\frac{1}{\langle t-r\rangle}|\Gamma^k vw+w\Gamma^k G w|
\langle r\rangle\langle t-r\rangle|\partial_r\Gamma^{k^{\prime}} G|\\
&\lesssim \frac{1}{(1+t)^{3/2}}E_k^{\frac{1}{2}}E_{k^{\prime}+3}^{\frac{1}{2}}E_k^{\frac{1}{2}}
+\frac{1}{1+t}E_k^{\frac{1}{2}}Z^{\frac{1}{2}}X_{k^{\prime}+3}^{\frac{1}{2}}\\
&\lesssim \frac{1}{(1+t)^{3/2}}E_k^{\frac{3}{2}}+\frac{1}{1+t}E_kZ^{\frac{1}{2}}.
\end{split}
\end{equation}
The third term is bounded by
\begin{equation}
\begin{split}
II_3&=\sum_{\substack{\beta+\gamma=\alpha, |\alpha|\leq k\\ \gamma\neq\alpha}} \int_{r\geq\frac{\langle t\rangle}{16}}
e^{-q}\Gamma^\alpha G_{ik}w_j\Gamma^\beta G_{jl}w_l(\partial_r\Gamma^\gamma G(I-w\otimes w))_{ik}\\
&\lesssim \int_{r\geq\frac{\langle t\rangle}{16}} |\Gamma^k G||w\Gamma^{k^\prime}G||\nabla\Gamma^{k-1}G|\\
&+
\sum_{\substack{\beta+\gamma=\alpha, |\alpha|\leq k\\ \gamma\neq\alpha}} \int_{r\geq\frac{\langle t\rangle}{16}}
|\Gamma^k G_{ik}w_j\Gamma^k G_{jl}w_l(\partial_r\Gamma^{k^{\prime}} G(I-w\otimes w))_{ik}|,
\end{split}
\end{equation}
using Lemma \ref{Decayr3/2}, we have
\begin{equation}
\begin{split}
&\ \ \ \ \int_{r\geq\frac{\langle t\rangle}{16}} |\Gamma^k G||w\Gamma^{k^\prime}G||\nabla\Gamma^{k-1}G|\\
&\lesssim \frac{1}{(1+t)^{3/2}}E_k^{\frac{3}{2}},
\end{split}
\end{equation}
on the other hand, in order to estimate the second term of $II_3$, we use the decomposition of $\nabla$
\eqref{Decompose} and the constraint of $G$ \eqref{ConstraintG},
\begin{equation}
\begin{split}
&\ \ \ \ \Gamma^k G_{ik}w_j\Gamma^k G_{jl}w_l(\partial_r\Gamma^{k^{\prime}} G(I-w\otimes w))_{ik}\\
&=\Gamma^k G_{ik}\Gamma^k G_{jl}w_l(w_j\partial_r\Gamma^{k^{\prime}} G_{ik}-w_j\partial_r\Gamma^{k^{\prime}}G_{ih}w_hw_k)\\
&=\Gamma^k G_{ik}\Gamma^k G_{jl}w_l(w_j\partial_r\Gamma^{k^{\prime}} G_{ik}-\partial_j\Gamma^{k^{\prime}}G_{ih}w_hw_k
-\frac{(w\wedge \Omega)_j}{r}\Gamma^{k^{\prime}}G_{ih}w_hw_k)\\
&=\Gamma^k G_{ik}\Gamma^k G_{jl}w_l(w_j\partial_r\Gamma^{k^{\prime}} G_{ik}-\partial_h\Gamma^{k^{\prime}}G_{ij}w_hw_k
-h_{k^{\prime}}w_hw_k-\frac{(w\wedge \Omega)_j}{r}\Gamma^{k^{\prime}}G_{ih}w_hw_k)\\
&=\Gamma^k G_{ik}\Gamma^k G_{jl}w_l(w_j\partial_r\Gamma^{k^{\prime}} G_{ik}-\partial_r\Gamma^{k^{\prime}}G_{ij}w_k
-h_{k^{\prime}}w_hw_k-\frac{(w\wedge \Omega)_j}{r}\Gamma^{k^{\prime}}G_{ih}w_hw_k),
\end{split}
\end{equation}
while
\begin{equation}
\begin{split}
&\ \ \ \ w_j\partial_r\Gamma^{k^{\prime}} G_{ik}-\partial_r\Gamma^{k^{\prime}}G_{ij}w_k\\
&=\partial_j\Gamma^{k^{\prime}} G_{ik}-\partial_k\Gamma^{k^{\prime}}G_{ij}+
\frac{(w\wedge \Omega)_j}{r}\Gamma^{k^{\prime}} G_{ik}-\frac{(w\wedge \Omega)_k}{r}\Gamma^{k^{\prime}} G_{ij}\\
&=h_{k^{\prime}}+
\frac{(w\wedge \Omega)_j}{r}\Gamma^{k^{\prime}} G_{ik}-\frac{(w\wedge \Omega)_k}{r}\Gamma^{k^{\prime}} G_{ij},
\end{split}
\end{equation}
hence
\begin{equation}
\begin{split}
&\ \ \ \ \sum_{\substack{\beta+\gamma=\alpha, |\alpha|\leq k\\ \gamma\neq\alpha}} \int_{r\geq\frac{\langle t\rangle}{16}}
|\Gamma^k G_{ik}w_j\Gamma^k G_{jl}w_l(\partial_r\Gamma^{k^{\prime}} G(I-w\otimes w))_{ik}|\\
&\lesssim \int_{r\geq\frac{\langle t\rangle}{16}}|\Gamma^k G||\Gamma^k G|(h_{k^{\prime}}+\frac{\Gamma^{k^\prime+1}G}{r})
\lesssim \frac{1}{(1+t)^2}E_k^{\frac{3}{2}}.
\end{split}
\end{equation}
The last term of $II$ is treat as follows
\begin{equation}
\begin{split}
II_4&=-\sum_{\substack{\beta+\gamma=\alpha, |\alpha|\leq k\\ \gamma\neq\alpha}} \int_{r\geq\frac{\langle t\rangle}{16}}
e^{-q}\Gamma^\alpha G_{ik}w_j\partial_r\Gamma^\gamma G_{il}w_l (\Gamma^\beta G(I-w\otimes w))_{jk}\\
&\lesssim\sum_{\substack{\beta+\gamma=\alpha, |\alpha|\leq k\\ \gamma\neq\alpha}} \int_{r\geq\frac{\langle t\rangle}{16}}
|\Gamma^\alpha G||\partial_r\Gamma^\gamma Gw||w\Gamma^\beta G(I-w\otimes w)|\\
&\lesssim \int_{r\geq\frac{\langle t\rangle}{16}}|\Gamma^k G||\partial_r\Gamma^{k-1} Gw||w\Gamma^{k^{\prime}} G(I-w\otimes w)|\\
&\ \ \ \ +\int_{r\geq\frac{\langle t\rangle}{16}}|\Gamma^k G||\partial_r\Gamma^{k^{\prime}} Gw||w\Gamma^k G(I-w\otimes w)|\\
&\lesssim \frac{1}{(1+t)^{3/2}}E_k^{\frac{1}{2}}E_k^{\frac{1}{2}}E_{k^{\prime}+2}^{\frac{1}{2}}
+\frac{1}{1+t}E_k^{\frac{1}{2}}X_{k^{\prime}+3}^{\frac{1}{2}}Z^{\frac{1}{2}}\\
&\lesssim \frac{1}{(1+t)^{3/2}}E_k^{\frac{3}{2}}+\frac{1}{1+t}E_kZ^{\frac{1}{2}}.
\end{split}
\end{equation}

It remains to treat the pressure term in \eqref{Gequation}. However, thanks to Lemma \ref{Pressure}, this term is handled
exactly as the preceding ones.

Finally, we gather our estimate for \eqref{Gequation} to get
\begin{equation}
\widetilde{E}_k^{\prime}(t)+Z
\leq\mu Z +C_\mu\frac{1}{(1+t)^{3/2}}E_k^{3/2},
\end{equation}
with
\begin{equation}
\widetilde{E}_k= \sum_{|\alpha|\leq k}\int e^{-q}(|\Gamma^\alpha v|^2 +|\Gamma^\alpha G|^2)dx.
\end{equation}
Notice that $E_k\backsim \widetilde{E}_k$, we obtain, for $\mu$ small
\begin{equation}
\widetilde{E}_k^{\prime}(t)
\leq C_\mu\frac{1}{(1+t)^{3/2}}\widetilde{E}_k^{3/2},
\end{equation}
This implies that $E_k(t)$ remains bounded by $M\epsilon^2$ for all time, here $M$ is
a positive constant which depends only on $k$.
Hence we complete the proof of Theorem \ref{Main}.

\section{General Isotropic Elastodynamics}
For general isotropic elastodynamics, the energy functional has the form
$W=W(F)$ with
\begin{equation}
W(F)=W(QF)=W(FQ),
\end{equation}
for all rotation matrices: $Q=Q^T$, $\text{det} Q=1$. The first relation is due to frame indifference,
while the second one expresses the isotropy of materials. This implies that $W$ depends
on $F$ through the principle invariants of $FF^T$, namely $\text{tr} FF^T$, $\frac{1}{2}[(\text{tr}FF^T)^2-\text{tr}((FF^T)^2)]$
and $\text{det} FF^T$ in 3D. Setting $\tau=\frac{1}{2}\text{tr} FF^T$,
$\gamma=\frac{1}{4}[(\text{tr}FF^T)^2-\text{tr}((FF^T)^2)]$
and $\delta=\text{det} F=(\text{det} FF^T)^{\frac{1}{2}}$, we may assume that $W(F)=\bar{W}(\tau, \gamma, \delta)$, for some smooth function
$\bar{W}: \mathbb{R}^+ \times \mathbb{R}^+\times \mathbb{R}^+\rightarrow \mathbb{R}^+$.
Since
\begin{equation}
\frac{\partial \tau}{\partial F}= F, \quad\frac{\partial \gamma}{\partial F} =\text{tr}(FF^T)F-FF^TF
\quad \text{and}\quad\frac{\partial \delta}{\partial F}= \delta F^{-T},
\end{equation}
the Piola-Kirchhoff stress has the form
\begin{equation}
\begin{split}
S(F)\equiv\frac{\partial W(F)}{\partial F}&= \bar{W}_\tau(\tau, \gamma, \delta)F+\bar{W}_\gamma(\tau, \gamma, \delta)
(\text{tr}(FF^T)F-FF^TF)\\
&\ \ \ \ +\bar{W}_\delta(\tau, \gamma, \delta)\delta F^{-T}.
\end{split}
\end{equation}
We assume that the reference configuration is stress free, $S(I)=0$, so that
\begin{equation}\label{restrict1}
\bar{W}_\tau(\frac{3}{2}, \frac{3}{2}, 1)+2\bar{W}_\gamma(\frac{3}{2}, \frac{3}{2}, 1)+\bar{W}_\delta(\frac{3}{2}, \frac{3}{2}, 1)=0.
\end{equation}
The Cauchy stress tensor is
\begin{equation}
\begin{split}
T(F)&\equiv \delta^{-1} S(F)F^T\\
&=\delta^{-1}\bar{W}_\tau(\tau, \gamma, \delta)FF^T+\delta^{-1}\bar{W}_\gamma(\tau, \gamma, \delta)
(\text{tr}(FF^T)FF^T-FF^TFF^T)\\
&\ \ \ \ +\bar{W}_\delta(\tau, \gamma, \delta)I,
\end{split}
\end{equation}
and the term $\nabla\cdot FF^T$ in \eqref{EquationM} is replaced by $\nabla\cdot T(F)$. Let us now proceed to examine
this term.

Write
\begin{equation}
\begin{split}
T(F)=& [\bar{W}_\tau(\frac{3}{2}, \frac{3}{2}, 1)+2\bar{W}_\gamma(\frac{3}{2}, \frac{3}{2}, 1)]FF^T\\
&+[\delta^{-1}\bar{W}_\tau(\tau, \gamma, \delta)-\bar{W}_\tau(\frac{3}{2}, \frac{3}{2}, 1)][FF^T-I]\\
&+[\delta^{-1}\bar{W}_\tau(\tau, \gamma, \delta)-\bar{W}_\tau(\frac{3}{2}, \frac{3}{2}, 1)+\bar{W}_\delta(\tau, \gamma, \delta)]I\\
&+\delta^{-1}\bar{W}_\gamma(\tau, \gamma, \delta)(\text{tr}(FF^T)FF^T-FF^TFF^T)-2\bar{W}_\gamma(\frac{3}{2}, \frac{3}{2}, 1)FF^T\\
&\equiv \sum_{a=1}^4 T_a(F).
\end{split}
\end{equation}

Assume that
\begin{equation}\label{restrict2}
\bar{W}_\tau(\frac{3}{2}, \frac{3}{2}, 1)+2\bar{W}_\gamma(\frac{3}{2}, \frac{3}{2}, 1)>0.
\end{equation}
Then $T_1(F)$ gives rise to a Hookean term. Notice that assumption \eqref{restrict2} rules out the hydrodynamical
case $\bar{W}_\delta=0$. The principal invariants can be expanded about the identity as
follows:
\begin{equation}\label{tau}
\tau =\frac{1}{2}\text{tr}FF^T=\frac{1}{2}\text{tr}(I+G)(I+G^T)=\frac{3}{2}+\text{tr}G+\frac{1}{2}\text{tr}GG^T,
\end{equation}
\begin{equation}\label{gamma}
\begin{split}
\gamma&=\frac{1}{4}[(\text{tr}(FF^T))^2-\text{tr}((FF^T)^2)]\\
&=\frac{1}{4}(2\tau)^2-\frac{1}{4}\text{tr}(I+G+G^T+GG^T)^2\\
&=\frac{3}{2}+\mathcal{O}(|G|^2),
\end{split}
\end{equation}
and
\begin{equation}\label{delta}
\delta=\text{det}F=\text{det}(I+G)=1+\text{tr}G+\frac{1}{2}[(\text{tr}G)^2-\text{tr}G^2]+\text{det}G.
\end{equation}
In the case of incompressible motion, we have $\delta=1$, so from \eqref{delta}, we get
\begin{equation}
\text{tr}G+\frac{1}{2}[(\text{tr}G)^2-\text{tr}G^2]+\text{det}G=0,
\end{equation}
and hence from \eqref{tau}
\begin{equation}
\tau-\frac{3}{2}=\mathcal{O}(|G|^2).
\end{equation}
Thus, we see that for $|G|\ll 1$
\begin{equation}
T_2(F)=\mathcal{O}[(\tau-\frac{3}{2})|G|+(\gamma-\frac{3}{2})|G|]=\mathcal{O}(|G|^3),
\end{equation}
which produces nonlinearities which are of cubic order or higher, while $T_3(F)$ leads to a
gradient term which can be included in the pressure. Similarly, the term $T_4(F)$ can
be written as a term $\mathcal{O}(|G|^3)$ plus a term  which can be included in the pressure.
The conclusion is that the general incompressible
isotropic case differs from the Hookean case by a nonlinear perturbation which is cubic
in the displacement gradient G. Such terms present no further obstacles in the proof of our theorem in 3D.
Hence, for general isotropic elastodynamics, we have the following theorem.
\begin{thm}
Let $(v_0, G_0)\in H^k_\Lambda$, with $k\geq 9$. Suppose that $(v_0, F_0) = (v_0, I+G_0)$ satisfy the constraints
\eqref{ConstraintG1} \eqref{ConstraintG2}, and $\|(v_0, G_0)\|_{H^k_\Lambda}<\epsilon$.
Assume that the smooth strain energy function $W(F)$ is isotropic, frame indifferent, and
satisfies \eqref{restrict1} \eqref{restrict2}.
Then there exist two positive constants $M$ and $\epsilon_0$ which depend only
on $k$ such that, if $\epsilon\leq\epsilon_0$, then the system of incompressible isotropic elastodynamics
\begin{equation}
\begin{cases}
\partial_tv + v\cdot\nabla v + \nabla p = \nabla\cdot T(F),\\
\partial_tF + v\cdot\nabla F = \nabla vF,\\
\nabla\cdot v = 0,\ \ \ \ \nabla\cdot F^T=0.
\end{cases}
\end{equation}
with initial data $(v_0, F_0) = (v_0, I+G_0)$ has a unique global solution $(v, F) = (v, I+G)$,
which satisfies $(v, G)\in H_\Gamma^k$ and $E_k^{1/2}(t)\leq M\epsilon$ for all $t\in [0, +\infty)$.

\end{thm}

%----------------------------------------------------------------------------acknowledgement
\section*{Acknowledgement}
The authors were supported by
NSFC (grant No.11171072, 11121101 and 11222107), NCET-12-0120, Shanghai Talent Development Fund and SGST 09DZ2272900.

%-----------------------------------------------------------------------------bibliography

\end{document}